\documentclass[11pt]{article}
\usepackage[margin=1in]{geometry}
\usepackage{setspace}
\usepackage{cprotect}
\usepackage{amsmath,amssymb}
\usepackage{dsfont}
\usepackage{amsthm}
\usepackage{bbm}
\usepackage{makeidx}
\usepackage{enumerate}
\usepackage{graphicx,float,psfrag,epsfig,caption,subcaption}
\usepackage{epstopdf}
\usepackage{amssymb}
\usepackage{mathrsfs}
\usepackage{color}
\usepackage[mathscr]{euscript}
\usepackage{algorithm,algpseudocode}
\usepackage[style=numeric-comp,maxbibnames=8,doi=false,isbn=false,url=false,eprint=false]{biblatex}
\newtheorem{thm}{Theorem}[section]
\newtheorem{prop}[thm]{Proposition}
\newtheorem{cor}{Corollary}[section]
\newtheorem{lem}[thm]{Lemma}

\theoremstyle{definition}
\newtheorem{defn}{Definition}[section]

\theoremstyle{remark}
\newtheorem*{rmk}{Remark}

\newcommand{\VAR}{{\mathrm{Var}}}
\newcommand{\PAR}[1]{{{\left(#1\right)}}}
\newcommand{\ABS}[1]{{{\left| #1 \right|}}} 
\algnewcommand\INPUT{\item[\textbf{Input:}]}%
\algnewcommand\OUTPUT{\item[\textbf{Output:}]}%

\title{Learning Sparse Graphons and the Generalized Kesten-Stigum Threshold}

\author{Emmanuel Abbe \footnote{EPFL and Princeton University} 
		\,\, Shuangping Li  \footnote{Princeton University}
		\,\, Allan Sly  \footnote{Princeton University}}
\date{}
\begin{document}
\maketitle
\begin{abstract}
The problem of learning graphons has attracted considerable attention across several scientific communities, with significant progress over the recent years in sparser regimes. Yet, the current techniques still require diverging degrees in order to succeed with efficient algorithms in the challenging cases where the local structure of the graph is homogeneous. This paper provides an efficient algorithm to learn graphons in the constant expected degree regime. The algorithm is shown to succeed in estimating the rank-$k$ projection of a graphon in the $L_2$ metric if the top $k$ eigenvalues of the graphon satisfy a generalized Kesten-Stigum condition.
\end{abstract}

\section{Introduction}
The problem of estimating graphons has been extensively studied in the statistics, mathematics, machine learning and network science literature \cite{choi2012stochastic,wolfe2013nonparametric,gao2015rate,klopp2017oracle,borgs2015private,borgs2015consistent,abbe2015recovering, mossel2015reconstruction,bickel2009nonparametric,bickel2011method,kallenberg1999multivariate,rohe2011spectral,abbe2015community,massoulie2014community,xu2014edge,airoldi2013stochastic,zhang2017estimating,anandkumar2014tensor,amini,degree_corr,AS3cpam,decelle2011asymptotic}. The term graphon is used here to designate the non-parametric function that governs the edge probabilities in an inhomogeneous or exchangeable random graph model. Namely, for a nonnegative symmetric measurable function $f$ on $[0,1]\times [0,1]$, a random graph $G=([n],E)$ is drawn by generating $X_1,\cdots,X_n$ i.i.d.\ Uniform on $[0,1]$, and by drawing the edges in $E$ independently conditionally on $X_1,\cdots,X_n$ such that 
\begin{align}
    \mathbb{P}(E_{ij}=1|X_1=x_1,\cdots,X_n=x_n)=\mathbb{P}(E_{ij}=1|X_i=x_i,X_j=x_j)=f(x_i,x_j).
\end{align}
The original works focused on the regime where the function $f$ is constant (i.e., not decaying with $n$), called the dense regime,  which corresponds to the regime where the graphon admits a topological interpretation in terms of graph limits \cite{borgs2008convergent, borgs2012convergent}. The ``sparse graphon'' setting has been used to refer to the case where
\begin{align}
    f=f_n=\rho_n Q, \quad Q: [0,1]\times [0,1] \to \mathbb{R}_+,
\end{align}
where $\rho_n$ is called the scale parameter or sparsity parameter, as it can be thought of as the probability with which one keeps edges independently in the graph (deleting others). A special case of interest is when the function $Q$ is piecewise constant, i.e,. stochastic block models (SBMs) \cite{holland,bui,dyer,boppana,bollo_inhomo,abbe2017community}, in which case the function $Q$ is equivalently defined by a $k\times k$ symmetric matrix where $k$ is the number of flats in $Q$.

The main question studied in the context of `learning' or `estimating' graphons at a scale $\rho_n$ is to estimate $Q$ up to some level of accuracy\footnote{E.g., in the $L_2$ metric and up to measure preserving maps.} given a single sample of the random graph $G$. 

Most of the papers focus on the regime where $\rho_n=d_n/n$ with diverging degrees $d_n$, and in particular on the statistical point of view, i.e., obtaining consistent estimators for $Q$ (up to measure-preserving maps) and convergence rates, without taking necessarily the computational efficiency into consideration. In particular, this is investigated by \cite{choi2012stochastic} for SBMs and by \cite{wolfe2013nonparametric,gao2015rate} under H\"{o}lder smoothness assumption and also SBMs, where \cite{wolfe2013nonparametric} shows the consistency and convergence rate of the maximum likelihood estimator  in the regime where $\rho_n=\log^3(n)/n$. \cite{gao2015rate} also considers  H\"{o}lder smoothness and SBMs and shows that a histogram approximation with least squares criterion achieves an optimal rate of convergence in the dense case where $\rho_n=1$. These results are further extended and improved in \cite{klopp2017oracle}, which obtains oracle inequalities for the moderately sparse regime where $\log(k)/n \wedge (k/n)^2 \le \rho_n \le \log(k)/{\sqrt{kn}} \wedge (k/n)^{3/2}$, where $k$ is the number of blocks for the block constant oracle. A similar result is also obtained in \cite{borgs2015private} for the case where the degrees diverge, with in addition a node-differentially private procedure to estimate the graphon. Further, \cite{borgs2015consistent} manages to remove the boundedness assumption when studying the least-square estimator for arbitrary integrable graphons with possibly long tails in their degree distribution (and more general latent spaces), showing also that the efficient algorithm based on degree sorting works whenever the underlying degree distribution is atomless. Finally, some papers show that in regimes of diverging degrees and diverging signal-to-noise ratio, it is possible to estimate efficiently and consistently the SBM parameters for two communities \cite{amini}, certain degree-corrected SBMs \cite{degree_corr}, and any identifiable SBM with diverging degrees and a finite number of communities \cite{abbe2015recovering}. 

To the best of our knowledge, none of the above approaches either work or have been proved to succeed in non-trivial cases where $\rho_n=1/n$, in particular for efficient algorithms. There are a few exceptions to that. In \cite{bickel2011method}, it is shown that for SBMs with different expected degrees across communities, it is possible to estimate the parameters of the graph using moments. This is however a regime where the local structure of the graph is non-homogeneous, i.e., dependent on the community membership, and where moment methods can be applied directly (see Section \ref{results} for further discussion). A new line of work has then been initiated in the homogeneous case, where the expected degrees are both constant in magnitude and across communities. In this case, \cite{mossel2015reconstruction} shows that the parameters can be estimated for the special case of two symmetric communities, extended in \cite{AS3cpam} for $k$ symmetric communities, as long as the model parameters are above the so-called Kesten-Stigum threshold (see Section \ref{results}). In particular, this is the regime subject to a sharp phase transition phenomenon since below the KS threshold, \cite{mossel2015reconstruction} shows that a consistent estimator cannot be obtained for two symmetric communities. A similar picture is conjectured for SBMs with several symmetric communities \cite{decelle2011asymptotic}, although it was shown in \cite{AS3cpam} that non-efficient estimators could be devised below the KS threshold starting from 4 communities. More generally, the question of estimating SBMs in the constant-degree regime with general (non-symmetric) $Q$'s remained open in this line of work, let alone the case of non-piecewise constant graphons.  

This paper tackles this open problem,  providing an efficient algorithm to learn constant-degree graphons (and SBMs) under a generalized KS threshold condition. We next describe the results and algorithm.

\section{Main Results}\label{results}


Because we have no information about the vertex indices, we can only hope to learn the graphon $Q$ up to a measure preserving map. The following is a classical distance used in this context. 
\begin{defn}
Let $\phi:[0,1]\rightarrow [0,1]$ be a measure preserving map, and for a function $W$ on $[0,1]\times [0,1]$, define $W^{\phi}(x,y)=W(\phi(x),\phi(y))$. We define the distance $\delta_2(W,W')=\inf\limits_{\phi_1,\phi_2}\|W^{\phi_1}-W'^{\phi_2}\|_2$, where we take the infimum over measure-preserving maps $\phi_i:[0,1]\rightarrow [0,1]$, $i=1,2$.
\end{defn}

To formulate our theorem, we make the following assumptions.

\noindent\textbf{Assumption 1.} The graphon $Q$ is bounded, i.e., we assume that $M:=\sup\limits_{(x,y)\in [0,1]^2} |Q(x,y)|$ is finite. 

Since $Q$ is bounded and symmetric, it admits an eigenvalue decomposition. We denote the eigenvalues in a descending order in absolute value: $\mu_1\geq |\mu_2|\geq |\mu_3|\geq \cdots $. Let $f_i$ be the corresponding eigenfunctions. Denote $Q_K$ to be a rank $K$ approximation of $Q$. In particular, $Q_K(x,y)=\sum\limits_{i=1}^K \mu_i f_i(x)f_i(y)$ is unique when $\mu_K\neq \mu_{K+1}$.

\noindent \textbf{Assumption 2.} We assume that $\int_0^1Q(x,y)dy$ takes the same constant value $q$ for every $x\in [0,1]$. This is the constant expected degree condition.

We are now ready to present our results. 
\begin{thm}
Let $Q$ satisfy Assumptions $1$ and $2$, let $r_0$ be such that $|\mu_{r_0}|>\sqrt{\mu_1}\geq |\mu_{r_0+1}|$, and assume that $\mu_i$ is simple for all $i \in [r_0]$. Then Algorithm $1$ produces an estimator $\hat Q$ such that\footnote{We use $ \stackrel{(P)}{\longrightarrow}$ for the convergence in probability as $n \to \infty$.}
\begin{align*}
\delta_2(\hat Q, Q_{r_0}) \stackrel{(P)}{\longrightarrow} 0.
\end{align*}
\end{thm}
In other words, if there are $r_0$  eigenvalues that lie outside the graphon's bulk and are simple, one can estimate efficiently and consistently the rank-$r_0$ projection of the graphon.  
\begin{rmk}
Algorithm 1 runs with high probability in time $O(n \log n)$.
\end{rmk}
\begin{cor}
Let $Q$ be low-rank and satisfy Assumptions $1$ and $2$. Assume that all nonzero $\mu_i$'s are simple and satisfy $|\mu_i|> \sqrt{\mu_1}$. Then Algorithm $1$ produces an estimator $\hat Q$ such that
\begin{align*}
    \delta_2(\hat Q, Q)\stackrel{(P)}{\longrightarrow}  0.
\end{align*}
\end{cor}
In other words, for a low-rank graphon $Q$, if all nonzero eigenvalues lie outside of the graphon's bulk and are simple, one can estimate $Q$ efficiently and consistently. In particular, this holds for constant expected degree stochastic block model when all nonzero eigenvalues are simple and satisfy $|\mu_i|>\sqrt{\mu_1}$.

For a graphon $Q$, and $h\in \mathbb{R}^+$, we define $\mathcal{Q}_h$ to be a scaled graphon such that $\mathcal{Q}_h(x,y)=h\cdot Q(x,y)$. 
\begin{cor}
Let $Q$ satisfy Assumptions $1$ and $2$ and assume that all nonzero $\mu_i$'s are simple. Then Algorithm $1$ produces an estimator $\widehat{\mathcal{Q}_h}$ such that
\begin{align*}
    \delta_2(\frac{1}{h}\widehat{\mathcal{Q}_h}, Q)\stackrel{(P)}{\longrightarrow} o_h(1).
\end{align*}
\end{cor}

\textbf{Discussions of assumptions.}
The condition $\int_0^1 Q(x,y)dy=q$ means that the random graph $G$ has a constant expected degree. This has been the most challenging regime in learning or community detection problems because in this case, degree of vertices or any simple acyclic subgraph counts yield no information towards communities.  Moreover, the Benjamini-Schramm limit of the graph is a Galton-Watson branching process independent of the vertex label at the root.

For a low rank graphon $Q$, the assumption that all nonzero eigenvalues satisfy $|\mu_i|>\sqrt{\mu_1}$ is a natural generalization of the Kesten-Stigum condition that requires a single eigenvalue to satisfy this inequality. First studied in the context of multi-type branching processes, it has played a central role in the study of the SBM.  In particular in the symmetric 2-block cases it determines the threshold for consistent estimation of the model and for weakly recovering the communities.  The threshold is equivalent to the eigenvalues of the non-backtracking matrix lying outside the bulk spectrum.

For the assumption of the simple eigenvalues, we believe that our algorithm can still generate a good estimator even if the eigenvalues are not simple. We have the assumption here for technical purposes. Proving the more general results requires more detailed analysis of the non-backtracking matrix.

\section{Algorithm}
Our algorithm makes use of the non-backtracking matrix $B$. The matrix has been used in various network problems including community detection of stochastic block models \cite{redemption,bordenave2018nonbacktracking}. The non-backtracking walk on a graph avoids returning to the vertex visited in the previous step, and can be viewed as the operator resulting from a linearization of the belief propagation algorithm \cite{redemption,powering1}. This additional requirement of ‘non-backtracking’ makes its spectrum behave well: in certain sparse cases, when $\rho_n=O(1/n)$, top eigenvalues of $B$ approximate top eigenvalues of $Q$. This is in contrast to the top eigenvectors of the adjacency matrix that concentrate around high degree vertices in the sparse regime, and top eigenvalue of order $(\log n / \log\log n)^{1/2}$ \cite{krivelevich2003largest}. Therefore classical spectral clustering algorithms using the adjacency matrix are not helpful in the sparse regime. In our problem, the non-backtracking matrix $B$ will be important in approximating eigenvalues and extracting moments of the eigenfunctions of $Q$. For a graph $G=(V,E)$, the non-backtracking matrix is indexed by the set of oriented edges $\vec E=\{(u,v):\{u,v\}\in E\}$, where each entry equals
\begin{align*}
    B_{ef}=\mathbbm{1}(e_2=f_1)\mathbbm{1}(e_1\neq f_2)=\mathbbm{1}(e_2=f_1)\mathbbm{1}(e\neq f^{-1}).
\end{align*}
Here for any $e=(u,v)\in \vec E$, we define $e_1=u$, $e_2=v$, $e^{-1}=(v,u)$. We define $\lambda_k$ to be the $k$-th eigenvalue of $B$ and $\xi_k$ to be the $k$-th normalized eigenvector of $B$. We order eigenvalues in descending order in magnitude and define $B_k(v)=\sum\limits_{e\in E_1:e_2=v}\xi_k(e)$.

Our algorithm consists of several steps:\\
1. We compute top eigenvalues of the non-backtracking matrix $B$. We then determine the approximate number of eigenvalues that have a magnitude larger than $\sqrt{\lambda_1}$. We call this number $K$.\\
2. We extract joint moments of $f_i$ through subgraph counts. For a vector $(i_1,\cdots,i_k)\in [K]^k$, we count a weighted sum of $k$-leaves stars in $G$, where the weight of each star is defined as $\prod_{j=1}^k B_{i_j}(v_j)$, where $v_i$ are leaves. This gives us an estimate of $\int_{[0,1]} f_{i_1}(x)\cdots f_{i_k}(x)dx$.\\
3. We consider the random variables $f_1(U),\cdots,f_K(U)$ where $U$ is Uniform on $[0,1]$. We use Lengendre polynomials to approximate their joint distribution function and compute the coefficients of the polynomials to match up the moment information from step 2.\\
4. We sample according to the approximated joint distribution of $f_1(U),\cdots,f_K(U)$. From there we obtain a piecewise constant approximation of $Q$.

\textbf{Some Discussions.} In our algorithm, we utilize the spectrum of the non-backtracking matrix to approximate eigenvalues of $Q$. There are actually a few other ways to do this. One of them is using cycle counts similarly to \cite{mossel2015reconstruction}, but stabilizing the counts to account for the different magnitudes of eigenvalues as in \cite{AS2nips}. We use here the spectrum of $B$ in our algorithm because we can get good controls of the error terms and because we anyway need eigenvectors of $B$.

To see why the subgraph counts gives us an estimate of the moment, we consider the inner product $\mathcal{C}_{ij}:=\sum_{v\in V} \frac{1}{\sqrt{n}} B_i(v)f_j(X_v)$. We will show that $\mathcal{C}_{ij}$ is asymptotically $\delta_{ij}$ after some normalization. We define $Q^X$ as the matrix such that $Q^X(u,v)=Q(X_u,X_v)$. Let $Q^X_u$ be the $u$-th column vector of $Q^X$. Similarly we define $f_i^X$ as the vector such that $f_i^X(v)=f_i(X_v)$. So the sum of $k$-leaf stars with weight $\prod_{j=1}^k B_{i_j}(v_j)$ is approximately
\begin{align*}
    &\sum_{u\in V} \prod_{j=1}^k B_{i_j}^t \frac{Q^X_{u}}{n} =n^{-k/2}\sum_{u\in V} \prod_{j=1}^k \frac{B_{i_j}^t}{\sqrt{n}} \sum_{\ell=1}^\infty \mu_\ell f_\ell^X f_\ell^X(u) \approx Cn^{-k/2}\sum_{u\in V} \prod_{j=1}^k \mu_{i_j} f_{i_j}(X_u)\\
    &\approx C n^{-k/2+1}\prod_{j=1}^k \mu_{i_j}\int_{[0,1]} f_{i_1}\cdots f_{i_k}dx.
\end{align*}
This allows us to estimate the joint moments of eigenfunctions of $Q$.

We pick Legendre polynomials to approximate our functions because they work well in modelling bounded support functions. See \cite{talenti1987recovering} for a one dimensional example of estimating functions with given moments. In our polynomial approximation step, there is one caveat in the computation. The joint distribution of $(f_1(U),\cdots,f_K(U))$ is degenerate: it is supported on a one dimensional curve. So we convolve it with a mollifier to make the distribution function well-defined on $\mathbb{R}^K$ and smooth. When the mollifier is small enough, the convolved distribution approximates the original function well.

In the following chart, we summarize our algorithm steps. We use $e_0$ to denote the error tolerance, i.e., with probability tending to $1$, $\delta_2(\hat Q-Q_{r_0})\leq e_0$ as $n$ goes to infinity.

\begin{algorithm}[H]
\caption{Sparse Graphon Estimation}
\label{alg:myalgo}
\begin{algorithmic}[1]
\INPUT $G$ (graph), $e_0>0$ (error tolerance), $M>0$ (bound on $Q$)
\OUTPUT $\hat Q$ (estimator)
\State $\epsilon \leftarrow \frac{1}{\log(\log (n))}$
\Function{Sample\_Splitting}{$\epsilon$,$E$}
    \State \Return $E_1$,$E_2$
\EndFunction
\Statex
\Function{Non\_Backtracking\_Matrix}{$V,E_1$}
    \State \Return $K,\lambda_1,\cdots,\lambda_K$, $B_1\cdots,B_K$
\EndFunction
\Statex
\State $N\leftarrow (\frac{2KM}{e_0})^{6K+30}$
\Function{Sub\_Graph\_Count}{$E_2,K,\lambda_1,\cdots,\lambda_K,B_1\cdots, B_K, N$}
    \State \Return $P_\alpha$ ($0\leq \alpha\leq \mathbf{N}$)
\EndFunction
\Statex
\State $\delta\leftarrow \sqrt{\frac{e_0}{64K\lambda_1 M^2}}$
\Function{Polynomial\_Approximation}{$P_\alpha,\delta,N$}
    \State \Return $\hat h_N^+$
\EndFunction
\Statex
\Function{Graphon\_Estimation}{$\hat h_N^+,K,\lambda_1,\cdots,\lambda_K$}
    \State \Return $\hat Q$
\EndFunction
\end{algorithmic}
\end{algorithm}

The following are detailed definitions of the functions.\\~\\
\textbf{Function} \textproc{Sample\_Splitting}($\epsilon$,$E$)\\
For each edge $e\in E$, mutually independently, with probability $1-\epsilon$, we assign it into $E_1$. Assign all the remaining edges to $E_2=E\backslash E_1$. We define $G_1=(V,E_1)$ and $G_2=(V,E_2)$.\\~\\
\textbf{Function} \textproc{Non\_Backtracking\_Matrix}($V,E_1$)\\
We compute the non-backtracking matrix of $G_1$, call it $B$. Take $e_1(n)=\frac{1}{\sqrt{\log n}}$. Define $K\in \mathbb{Z}_+$ to be such that $|\lambda_i|>\sqrt{\lambda_1}+e_1(n)$ for all $i\in [K]$ and $|\lambda_{K+1}|\leq \sqrt{\lambda_1}+e_1(n)$. We compute eigenvalues $\lambda_k$ of $B$ in a descending order of magnitude, until we reach $\lambda_K$. Compute the corresponding eigenvectors $\xi_k$. For $k\in [K]$, compute $B_k(v)=\sum\limits_{e\in E_1:e_2=v}\xi_k(e)$.\\
\\~\\
\textbf{Function} \textproc{Sub\_Graph\_Count}($E_2,K,\lambda_1,\cdots,\lambda_K,B_1\cdots, B_K, N$)\\
For $1\leq i,j\leq n$, we write $i\sim  j$ whenever there is an edge in $E_2$ that connects $i$ and $j$. For $k\in [K]$, define
\begin{align*}
&A_{kk}:=\sum\limits_{1\leq i,j\leq n}B_k(i)B_k(j)\mathbbm{1}[i\sim  j],\\
&P_{kk}:=\frac{A_{kk}}{\epsilon \lambda_k}.
\end{align*}
For any multi-index $\alpha\in \mathbb{N}^K$, i.e. $\alpha=(\alpha_1, \alpha_2,\cdots, \alpha_K)$, where each $\alpha_i\in \mathbb{N}$, we define $|\alpha|=\sum_{i=1}^K \alpha_i$. We use $[n]^k_\leq$ to denote the set consisting of all ordered $k$-tuples from the set $[n]$. More precisely, we define
\begin{align*}
    [n]^k_\leq=\{(i_1,\cdots,i_k)\in [n]^k:1\leq i_1\leq\cdots\leq i_k\leq n\}.
\end{align*}
For a multi-index $\alpha$, we define $I^\alpha =(I^\alpha_1,\cdots,I^\alpha_{|\alpha|}) \in [K]^{|\alpha|}_\leq$ such that $|\{1\leq j\leq |\alpha|:I^\alpha_j=i\}|=\alpha_i$ for any $1\leq i\leq K$. We define
\begin{align*}
    &A_{\alpha}:=\sum_{1\leq i_1,\cdots,i_{|\alpha|},w\leq n}\prod_{\ell=1}^{|\alpha|} B_{I^{\alpha}_\ell}(i_\ell)\mathbbm{1}[i_\ell\sim w],
\end{align*}
where all $i_\ell$ are disjoint. Define
\begin{align*}
    P_{\alpha}:=\frac{A_{\alpha}n^{|\alpha|/2-1}}{\epsilon^{|\alpha|} \prod_{i=1}^K (\sqrt{P_{ii}} \lambda_i)^{\alpha_i}},
\end{align*}
when $P_{ii}>0$, for all $i\in [K]$. Otherwise, set all $P_\alpha=0$. Define a multi-index $\mathbf{N}=(N,N,\cdots,N).$
For any multi-index $0\leq \alpha\leq \mathbf{N}$, we compute $P_\alpha$.\\~\\
\textbf{Function}
\textproc{Polynomial\_Approximation}($P_\alpha,\delta,N$)\\
We use the following bump function as our mollifier. Define a bump function $\Psi_\delta:\mathbb{R}\rightarrow \mathbb{R}$ to be
\begin{align*}
\Psi_\delta(x)=\begin{cases}
        \exp(-\frac{1}{\delta^2-x^2}), &x\in (-\delta,\delta)\\
            0, &\text{otherwise}.
        \end{cases}
\end{align*}
For any multi-index $0\leq \alpha\leq \mathbf{N}$, we compute
\begin{align*}
    M_{\alpha}(\delta)=\sum\limits_{0\leq \beta\leq \alpha}P_\beta\prod_{i=1}^K {\alpha_i \choose \beta_i}\mathbb{E}[N_\delta^{\alpha_i-\beta_i}],
\end{align*}
where $N_\delta$ is a random variable whose density function is proportional to $\Psi_\delta(x)$.\\~\\
We use Legendre polynomials to approximate our mollified distribution. Let $L_0(x),L_1(x),\cdots,L_i(x),\cdots$ be Legendre polynomials, normalized by $\int_{-1}^1 (L_i(x))^2 dx=1$. Recall that these polynomials are defined uniquely (apart from sign) by the following requirements:
\begin{align*}
    &L_i(x) \text{ is a polynomial }\\
    & deg( L_i(x) )=i\\
    &\int_{-1}^1 L_i(x)L_j(x) dx=\delta_{ij}.
\end{align*}
We also define Legendre polynomials in dimension $K$ as follows. For a multi-index $\alpha$, define
\begin{align*}
    L_\alpha(x_1,\cdots,x_K)=\prod_{i=1}^K L_{\alpha_i}(x_i).
\end{align*}
They are polynomials of degree $|\alpha|$ and satisfy
\begin{align*}
   \int_{[-1,1]^K} L_\alpha(x_1,\cdots,x_K)L_\beta(x_1,\cdots,x_K)d\mathbf{x} =\delta_{\alpha,\beta}.
\end{align*}
Define $C_{ij}$ to be the coefficients of Legendre polynomials:
\begin{align*}
    \sum\limits_{j=0}^i C_{ij}x^j=L_i(x).
\end{align*}
Define $C^{\otimes K}$ to be the Kronecker product of $K$ copies of $C$, i.e., $C^{\otimes K}_{\alpha,\beta}=\prod_{i=1}^K C_{\alpha_i,\beta_i}$. Now, we consider scaled Legendre polynomials on $[-\frac{2M}{\sqrt{\lambda_1}},\frac{2M}{\sqrt{\lambda_1}}]$, satisfying the same conditions as above except the integration is over $[-\frac{2M}{\sqrt{\lambda_1}},\frac{2M}{\sqrt{\lambda_1}}]$. Denote them as $\tilde L$. Write $\kappa=\frac{2M}{\sqrt{\lambda_1}}$. Then,
\begin{align*}
    \tilde L_\alpha(x_1,\cdots,x_K)= \frac{1}{\kappa^{K/2}} L_\alpha(x_1/\kappa,\cdots,x_K/\kappa).
\end{align*}
Define $\tilde{C}$ such that $\tilde{C}_{ij}=C_{ij}/\kappa^{j+1/2}$. Define $\tilde C^{\otimes K}$ to be the Kronecker product of $K$ copies of $\tilde C$. Now we compute
\begin{align*}
    \hat\rho_\alpha(\delta)=\sum\limits_{0\leq \beta\leq \alpha} \tilde C^{\otimes K}_{\beta}M_{\beta}(\delta).
\end{align*}
Finally, for a given $N\in \mathbb{N}_+$, compute
\begin{align*}
    \hat h_N(x_1,\cdots,x_K)=\sum\limits_{0\leq \alpha\leq \mathbf{N}}\hat\rho_\alpha(\delta) \tilde L_\alpha(x_1,\cdots,x_K).
\end{align*}
Define $\hat h_N^+:=\hat h_N \mathbbm{1}[\hat h_N\geq 0]$.\\~\\
\textbf{Function} \textproc{Graphon\_Estimation}($\hat h_N^+,K,\lambda_1,\cdots,\lambda_K$)\\
Take $m=n$. We sample vectors $Z_1,\cdots,Z_m$ from a distribution proportional to $\hat h_N^+$ mutually independently. For any $1\leq i\leq K$, compute
\begin{align*}
    \hat f_i(x)=Z_{\lceil xm\rceil }(i).
\end{align*}
We compute
\begin{align*}
    \hat Q(x,y)=\sum\limits_{i=1}^K \lambda_i \hat f_i(x)\hat f_i(y).
\end{align*}

\textbf{Complexity.} With high probability, the running time of our algorithm is $O(n\log n)$. The non-backtracking matrix $B$ is typically sparse with $O(n)$ non-zero entries. Therefore, extracting $K$ largest magnitude eigenvalues and corresponding eigenvectors takes time $O(n\log n)$ by the power algorithm. This implies that the complexity of \textproc{Non\_Backtracking\_Matrix} is with high probability $O(n\log n)$. The rest of the algorithm runs in time $O(n)$. 

\section{Proof Outline}
In this section, we list a few key lemmas for the proof of our main result. The auxiliary lemmas and the proof of the key lemmas are deferred to later sections.

\begin{lem}\label{l:BLM}
Let $Q$ satisfy Assumptions $1$ and $2$. Then with probability tending to $1$ as $n\rightarrow \infty$,
\begin{align*}
    \lambda_k(B)=\mu_k+O(\frac{1}{\log n}), \text{ for } k\in [r_0] \text{,  and for } k>r_0, |\lambda_k(B)|\leq \sqrt{\mu_1} +O(\frac{1}{\log n}).
\end{align*}
Recall that $\mathcal{C}_{ij}:=\frac{1}{\sqrt{n}}\sum\limits_{\ell=1}^n B_i(\ell)f_j(X_\ell).$ We also have that for $i\in [r_0]$,
\begin{align*}
    \mathcal{C}_{ii} \stackrel{(P)}{\longrightarrow} \theta_i,
\end{align*}
and
\begin{align*}
    \sum_{\ell=1}^\infty \mu_\ell\mathcal{C}_{i\ell}^2 \stackrel{(P)}{\longrightarrow} \mu_i \theta_i^2,
\end{align*}
where $\theta_i$ is a nonzero constant.
\end{lem}
The result on eigenvalues is a direct generalization of Theorem $4$ in \cite{bordenave2018nonbacktracking}. To prove the convergence of $\mathcal{C}_{ij}$, we also need to use results in \cite{bordenave2018nonbacktracking} and generalized versions for graphons. More specifically, we need results that relate the eigenvectors of the nonbacktracking matrix $B$ to some local statistics of the random graph. 
This is then related to certain functionals on branching processes via coupling. 

In particular, Lemma \ref{l:BLM} immediately implies that
\begin{lem}
With probability tending to $1$ as $n\rightarrow \infty$,
\begin{align*}
    K=r_0.
\end{align*}
\end{lem}

The next few lemmas are about counting the weighted stars described in Section 3.
\begin{lem}
For $i\in [r_0]$,
\begin{align*}
    P_{ii}=\frac{A_{ii}}{\epsilon \lambda_i} \stackrel{(P)}{\longrightarrow} \theta_i^2.
\end{align*}
\end{lem}
Through a second moment argument, we can show that $A_{ii} \approx \epsilon B_i^t Q^X B_i$. Then the proof of this lemma follows by combining the estimation on $\mathcal{C}_{ij}$ and $\lambda_i$.

\begin{lem}\label{l:moment}
For $0\leq \alpha\leq \mathbf{N}$, we have that
\begin{align*}
    P_\alpha=\frac{A_{\alpha}n^{|\alpha|/2-1}}{\epsilon^{|\alpha|} \prod_{i=1}^K (\sqrt{P_{ii}} \lambda_i)^{\alpha_i}}\stackrel{(P)}{\longrightarrow} \int_{[0,1]} f_\alpha dx.
\end{align*}
\end{lem}
Similar to the above, by a second moment argument, we can show that $A_{\alpha} \approx \sum_{w=1}^n\prod_{\ell=1}^{|\alpha|} \epsilon B_{I^{\alpha}_\ell}^tQ^X(w)$. So the lemma follows by combining estimation on $\mathcal{C}_{ij}$, $P_{ii}$ and $\lambda_i$.

Let $U$ be Uniform on $[0,1]$ and let $N_i$ be a series of iid random variables whose distribution function is proportional to $\Psi_\delta(x)$. Define $Y_i=f_i(U)+N_i$. Then the joint distribution of $(Y_1,\cdots,Y_K)$ is the same as $(f_1(U)+N_1,\cdots, f_K(U)+N_K)$. For each multi-index $\alpha=(\alpha_1,\cdots,\alpha_K)$, we write $f_\alpha = \prod_{i=1}^{|\alpha|} f_{i}^{\alpha_i}.$ Therefore,
\begin{align*}
    \int f_\alpha dx
    &=\mathbb{E}[f_{1}(U)^{\alpha_1}\cdots f_{K}(U)^{\alpha_K}]=\mathbb{E}[f_\alpha(U)].
\end{align*}
We can also relate this term to the joint moments of the $Y_i$'s.
\begin{align*}
    \mathbb{E}[Y_{1}^{\alpha_1}\cdots Y_{K}^{\alpha_K}]
    &=\mathbb{E}[(f_{1}(U)+N_1)^{\alpha_1}\cdots (f_{K}(U)+N_K)^{\alpha_K}]\\
    &=\sum\limits_{0\leq \beta\leq \alpha}\mathbb{E}[f_\beta(U)]\prod_{i=1}^K {\alpha_i \choose \beta_i}\mathbb{E}[N_\delta^{\alpha_i-\beta_i}]\\
    &=:\mu_\alpha(\delta),
\end{align*}
where $N_\delta$ is a random variable whose distribution function is proportional to $\Psi_\delta(x)$. Lemma \ref{l:moment} makes sure that the function \textproc{Sub\_Graph\_Count}($E_2,K,\lambda_1,\cdots,\lambda_K,B_1\cdots, B_K, N$) outputs a good estimation of $\mathbb{E}[f_\alpha(U)]$. This implies that $\mu_\alpha(\delta) \approx M_\alpha(\delta)$. Further, we have the lemma below.

Let $u$ be the joint distribution function of $(f_1(U)+N_1, \cdots,f_K(U)+N_K)$.
\begin{lem}
With probability tending to $1$ as $n\rightarrow \infty$,
\begin{align*}
    \int_{[-\kappa,\kappa]^K} \left|\frac{\hat h_N^+}{\|\hat h_N^+\|_1}-u\right| d\mathbf{x} \leq \frac{e_0}{2^5K\lambda_1^2 \kappa^4}.
\end{align*}
\end{lem}
This lemma shows that given approximate moments, multivariate polynomials up to degree $\mathbf{N}$ can form good estimates of the joint distribution function. This is proved by controlling the tail terms when the joint distribution is expanded in the Legendre polynomial basis. This term is related to the smoothness properties of $u$, which is why the mollifying step is necessary in our algorithm. 

\begin{lem}
Let $Z_1,\cdots,Z_m$ be iid random vectors following the distribution proportional to $\hat h_N^+$. Then there exists a measure preserving map $\phi:[0,1]\rightarrow [0,1]$, such that 
\begin{align*}
   \int_{[0,1]}|Z_{\lceil \phi(x)m\rceil }(i)-f_i(x)|^2 dx \leq \frac{e_0}{4K\lambda_1^2 \kappa^2}
\end{align*}
with probability tending to $1$ and $m\rightarrow \infty$.
\end{lem}
This lemma together with our estimation on $\lambda_i$ gives the desired result: with probability tending to $1$ and $n \rightarrow \infty$, $\delta_2(\hat Q,Q_{r_0})\leq e_0$. 

\section{Proofs}
\subsection{Spectrum of the Non-backtracking Matrix}
\begin{thm}[Theorem $4$ in \cite{bordenave2018nonbacktracking}]\label{t:eigSBM}
Let $Q$ be an SBM satisfying the constant expected degree condition. Then with probability tending to $1$ as $n\rightarrow \infty$,
\begin{align*}
    \lambda_k(B)=\mu_k+o(1), \text{ for } k\in [r_0] \text{,  and for } k>r_0, |\lambda_k(B)|\leq \sqrt{\mu_1} +o(1).
\end{align*}
\end{thm}
Actually, from Proposition $8$, $19$ and $20$ in the text, we have that
\begin{align*}
    \lambda_k(B)=\mu_k+O(\frac{1}{\log n}), \text{ for } k\in [r_0] \text{,  and for } k>r_0, |\lambda_k(B)|\leq \sqrt{\mu_1} +O(\frac{1}{\log n}).
\end{align*}

\begin{thm}\label{t:eig}
Let $Q$ be a graphon satisfying Assumptions $1$ and $2$. Then with probability tending to $1$ as $n\rightarrow \infty$,
\begin{align*}
    \lambda_k(B)=\mu_k+O(\frac{1}{\log n}), \text{ for } k\in [r_0] \text{,  and for } k>r_0, |\lambda_k(B)|\leq \sqrt{\mu_1} +O(\frac{1}{\log n}).
\end{align*}
\end{thm}
The proof is a direct generalization of Theorem \ref{t:eigSBM}. So we omit it here. This theorem immediately implies the following lemma.

\begin{lem}\label{kr0}
With probability tending to $1$ as $n\rightarrow \infty$,
\begin{align*}
    K=r_0.
\end{align*}
\end{lem}

\subsection{Subgraph Counts}
Recall that
\begin{align*}
    \mathcal{C}_{ij}:=\frac{1}{\sqrt{n}}\sum\limits_{\ell=1}^n B_i(V_\ell)f_j(X_\ell).
\end{align*}
\begin{lem}\label{l:cconverge}
For $i\in [r_0]$,
\begin{align*}
    \mathcal{C}_{ii} \stackrel{(P)}{\longrightarrow} \theta_i,
\end{align*}
and
\begin{align*}
    \sum_{\ell=1}^\infty \mu_\ell\mathcal{C}_{i \ell}^2 \stackrel{(P)}{\longrightarrow} \mu_i \theta_i^2.
\end{align*}
as $n$ goes to infinity, where $\theta_i$ is a nonzero constant.
\end{lem}
We will prove this lemma in later sections. As all $\mu_i$'s are simple for $i\in [r_0]$, corresponding eigenfunctions are unique up to signs. Notice that if we reverse the sign of $f_i$, then the sign of $\theta_i$ is also reversed. So there is an orientation of $f_1,\cdots,f_{r_0}$ such that all the corresponding $\theta_i$'s are positive. From now on, without loss of generality, we fix such orientation.

From now on, we work in the case when $K=r_0$. Recall from our definition that for $k\in [K]$,
\begin{align*}
    A_{kk}:=\sum\limits_{1\leq i,j\leq n}B_k(i)B_k(j)\mathbbm{1}[i\sim  j].
\end{align*}
Given $X_1,\cdots,X_n$, denote $Q^X$ as a matrix where its $(i,j)$-th entry equals to $Q(X_i,X_j)$.
\begin{lem}
Fix $E_1$. For $k\in [K]$ and $0<\gamma<\frac{1}{2}$, with probability larger than $1-2 \|B_k\|_2^4 \epsilon M n^{-2\gamma}$,
\begin{align*}
    |A_{kk}-\epsilon B_k^t Q^X B_k|\leq 2n^{-\frac{1}{2}+\gamma}.
\end{align*}
\end{lem}
\begin{proof}
Recall that vertices $i$ and $j$ are connected by an edge with probability $\epsilon Q^X(i,j)/n$. Define $D_{ij}=B_k(i)B_k(j)\mathbbm{1}[i\sim  j]$. Then $D_{ij}$ and $D_{i'j'}$ are independent whenever $i\sim j$ and ${i'}\sim {j'}$ are different edges. Thus
\begin{align*}
    \VAR(A_{kk}) &= \VAR\left(2\sum\limits_{1\leq i<j\leq n}B_k(i)B_k(j)\mathbbm{1}[i\sim  j]\right)\\
    &\leq 4\sum\limits_{1\leq i<j\leq n}B_k(i)^2B_k(j)^2\frac{\epsilon M}{n}\\
    &\leq 2 \|B_k\|_2^4 \frac{\epsilon M}{n}.
\end{align*}
So by a second moment estimate,
\begin{align*}
    \mathbb{P}[|A_{kk}-\mathbb{E}[A_{kk}]|\geq n^{-\frac{1}{2}+\gamma}] \leq \frac{\VAR(A_{kk})}{n^{-1+2\gamma}}\leq 2\|B_k\|_2^4 \epsilon M n^{-2\gamma}.
\end{align*}
As $|\mathbb{E}[A_{kk}]-\epsilon B_k^t Q^X B_k|=o(n^{-\frac{1}{2}+\gamma})$, the statement follows.
\end{proof}

By Lemma \ref{l:cconverge}, we have
\begin{align*}
    B_k^t Q^X B_k&=  \sum_{i,j=1}^n B_k(i) Q^X(i,j) B_k(j)\\
    &=\sum_{i,j=1}^n B_k(i) \left (\sum_{\ell=1}^\infty \mu_\ell f_\ell(X_i)f_\ell(X_j) \right )B_k(j) \\
    &= \sum_{\ell=1}^\infty \mu_\ell \mathcal{C}^2_{k\ell}\\
    &\stackrel{(P)}{\longrightarrow} \mu_k \theta_k^2,
\end{align*}
as $n$ goes to infinity. We define $d_1$ as the maximum degree of vertices in $G_1$. Then $d_1 =O(\log n)$ with high probability. This implies that $\|B_i\|_2=O(\log n)$. So $\epsilon^{-1}A_{kk}$ converges to $\mu_k\theta_k^2$ in probability. By Theorem \ref{t:eig}, $P_{kk}$ converges to $\theta_k^2$ in probability. Recall that we use $Q^X_w$ to denote the $w$-th column of $Q^X$.
\begin{lem}
Fix $E_1$. For any multi-index $\alpha$ and $0<\gamma<\frac{1}{2}$, with probability larger than $1-2|\alpha|!(\epsilon M)^{|\alpha|}\prod_{\ell=1}^K \|B_{I_\ell^\alpha}\|_2^{2\alpha_i} n^{-2\gamma}$, we have that
\begin{align*}
    |A_{\alpha}-\sum_{w=1}^n\prod_{\ell=1}^{|\alpha|} \epsilon B_{I^{\alpha}_\ell}^tQ^X(w)| \leq 2n^{-\frac{|\alpha|-1}{2}+\gamma}.
\end{align*}
\end{lem}
\begin{proof}
Recall that 
\begin{align*}
    A_{\alpha}:=\sum_{1\leq i_1,\cdots,i_{|\alpha|},w\leq n}\prod_{\ell=1}^{|\alpha|} B_{I^{\alpha}_\ell}({i_\ell})\mathbbm{1}[{i_\ell}\sim w],
\end{align*}
where all $i_\ell$ are disjoint. Define $D_{i,w}:=\prod_{\ell=1}^{|\alpha|} B_{I^{\alpha}_\ell}({i_\ell})\mathbbm{1}[{i_\ell}\sim w]$. Define $E(D_{i,w})$ to be the set of all edges ${i_\ell}\sim w$. We use $[n]^k_<$ to denote the set consisting of all strictly ordered $k$-tuples from the set $[n]$. More precisely, we define
\begin{align*}
    [n]^k_<=\{(i_1,\cdots,i_k)\in [n]^k:1\leq i_1<\cdots<i_k\leq n\}.
\end{align*}
Thus we can also write $A_\alpha$ as 
\begin{align*}
    A_{\alpha}&=|\alpha|!\sum_{i\in [n]^{|\alpha|}_<, w\in [n]} \prod_{\ell=1}^{|\alpha|} B_{I^{\alpha}_\ell}({i_\ell})\mathbbm{1}[{i_\ell}\sim w]\\
    &=|\alpha|!\sum_{i\in [n]^{|\alpha|}_<, w\in [n]} D_{i,w}
\end{align*}
Note that $\prod_{\ell=1}^{|\alpha|} \mathbbm{1}[{i_\ell}\sim w]$ follows Bernoulli $(\prod_{\ell=1}^{|\alpha|}\frac{\epsilon Q^X(i_\ell,w)}{n})$ distribution. So variance of $D_{i,w}$ is bounded by $\prod_{\ell=1}^{|\alpha|}B_{I_\ell^\alpha}({i_\ell})^2\frac{\epsilon M}{n} $. Also, we notice that $D_{i,w}$ and $D_{i',w'}$ are independent whenever $E(D_{i,w})$ and $E(D_{i',w'})$ are disjoint. When they are not, we write $(i,w)\sim^a (i',w')$, if $E(D_{i,w})$ and $E(D_{i',w'})$ 
have $a$ edges in common. If $(i,w)\sim^a (i',w')$, then $|\mathrm{Cov}(D_{i,w},D_{i',w'})|\leq (\frac{\epsilon M}{n})^{2|\alpha|-a} \prod_{\ell=1}^{|\alpha|}|B_{I^\alpha_\ell}({i_\ell})B_{I^\alpha_\ell}({i'_\ell})|$. Then we have that
\begin{align*}
    &\VAR(A_\alpha)\\
    &= (|\alpha|!)^2\sum_{i\in [n]^{|\alpha|}_<, w \in [n]} \VAR(D_{i,w})+ (|\alpha|!)^2\sum_{i,i'\in [n]^{|\alpha|}_<, w,w' \in [n], (i,w)\sim (i',w')} \mathrm{Cov}(D_{i,w},D_{i',w'})\\
    &\leq |\alpha|!\sum_{i\in [n]^{|\alpha|}, w \in [n]} \prod_{\ell=1}^{|\alpha|}B_{I_\ell^\alpha}({i_\ell})^2\frac{\epsilon M}{n}\\
    &\quad +\sum_{a=1}^{|\alpha|-1} \sum_{i,i'\in [n]^{|\alpha|}, w,w' \in [n], (i,w)\sim^a (i',w')} (\frac{\epsilon M}{n})^{2|\alpha|-a} \prod_{\ell=1}^{|\alpha|}|B_{I^\alpha_\ell}({i_\ell})B_{I^\alpha_\ell}({i'_\ell})|\\
    &\leq |\alpha|!\frac{(\epsilon M)^{|\alpha|}}{n^{|\alpha|-1}}\prod_{\ell=1}^K \|B_{I_\ell^\alpha}\|_2^{2\alpha_\ell}
    +\sum_{a=1}^{|\alpha|-1} C_a\frac{(\epsilon M)^{2|\alpha|-a}}{n^{|\alpha|-1}}\prod_{\ell=1}^K\|B_{I_\ell^\alpha}\|_2^{2\alpha_\ell}
    +C\frac{(\epsilon M)^{2|\alpha|-1}}{n^{2|\alpha|-1}}\prod_{\ell=1}^K\|B_{I_\ell^\alpha}\|_1^{2\alpha_\ell}\\
    &\leq 2|\alpha|!\frac{(\epsilon M)^{|\alpha|}}{n^{|\alpha|-1}}\prod_{\ell=1}^K \|B_{I_\ell^\alpha}\|_2^{2\alpha_\ell},
\end{align*}
where $C_a$ and $C$ are some constants. The second last inequality follows by separating cases into whether $w=w'$ and that $\|B_i\|_1^2\leq n \|B_i\|_2^2$. So by a second moment estimate, we have that
\begin{align*}
    \mathbb{P}[|A_\alpha-\mathbb{E}[A_\alpha]|\geq n^{-\frac{|\alpha|-1}{2}+\gamma} ]&\leq \frac{2|\alpha|!\frac{(\epsilon M)^{|\alpha|}}{n^{|\alpha|-1}}\prod_{\ell=1}^K \|B_{I_\ell^\alpha}\|_2^{2\alpha_\ell}}{n^{-|\alpha|+1+2\gamma}}\\
    &\leq 2|\alpha|!(\epsilon M)^{|\alpha|}\prod_{\ell=1}^K \|B_{I_\ell^\alpha}\|_2^{2\alpha_\ell} n^{-2\gamma}.
\end{align*}
As $|\mathbb{E}[A_\alpha]-\sum_{w=1}^n\prod_{\ell=1}^{|\alpha|} \epsilon B_{I^{\alpha}_\ell}^tQ^X(w)|=o(n^{-\frac{|\alpha|-1}{2}+\gamma})$, the statement follows.
\end{proof}
This lemma implies that with probability larger than $1-O(n^{-\gamma})$, we have that, for any $0\leq \alpha\leq \mathbf{N}$,
\begin{align*}
    |n^{\frac{|\alpha|}{2}-1}\epsilon^{-|\alpha|}A_{\alpha}- n^{\frac{|\alpha|}{2}-1}\sum_{w=1}^n\prod_{\ell=1}^{|\alpha|}  B_{I^{\alpha}_\ell}^tQ^X(w)| \leq 2n^{-\frac{1}{2}+\gamma}\epsilon^{-|\alpha|}.
\end{align*}
Therefore, for any $0\leq \alpha\leq \mathbf{N}$,
\begin{align*}
    n^{\frac{|\alpha|}{2}-1}\epsilon^{-|\alpha|}A_{\alpha} \stackrel{(P)}{\longrightarrow} n^{\frac{|\alpha|}{2}-1}\sum_{w=1}^n\prod_{\ell=1}^{|\alpha|}  B_{I^{\alpha}_\ell}^tQ^X(w),
\end{align*}
as $n$ goes to infinity.

Recall that, for each multi-index $\alpha=(\alpha_1,\cdots,\alpha_K)$, we write $f_\alpha = \prod_{i=1}^{|\alpha|} f_{i}^{\alpha_i}.$
\begin{lem}
For $0\leq \alpha\leq \mathbf{N}$, we have that
\begin{align*}
    n^{\frac{|\alpha|}{2}-1}\epsilon^{-|\alpha|}A_{\alpha}\stackrel{(P)}{\longrightarrow} \prod_{i=1}^K (\theta_i \mu_i)^{\alpha_i} \int_{[0,1]} f_\alpha(x) dx,
\end{align*}
as $n$ goes to infinity.
\end{lem}
\begin{proof}
Define $A_j(X_w)=\sum_{\ell=1}^\infty \mu_\ell f_\ell(X_w) \mathcal{C}_{j \ell}$ and $B_j(X_w)=\mu_{_j} f_{j}(X_w) \theta_{j}$. By Lemma \ref{l:cconverge}, we have that for $j\in [K]$,
\begin{align*}
    |A_j(X_w)-B_j(X_w)|&\leq \sum_{\ell=1}^\infty \mu_\ell f_\ell(X_w) |\mathcal{C}_{j \ell}-\delta_{j,\ell}\theta_{j}|\\
    &\leq \left(M \sum_{\ell=1}^\infty \mu_\ell |\mathcal{C}_{j \ell}-\delta_{j,\ell}\theta_{j}|^2\right)^{\frac{1}{2}}\\
    &\leq \left(M |\sum_{\ell=1}^\infty \mu_\ell \mathcal{C}_{j \ell}^2 -\mu_{j}\theta_{j}^2|+M \mu_{j} |\mathcal{C}_{jj}-\theta_{j}|^2\right)^{\frac{1}{2}},
\end{align*}
which goes to zero in probability. We also note that with high probability,
\begin{align*}
    |A_j(X_w)|^2\leq K \sum_{\ell=1}^\infty \mu_\ell \mathcal{C}_{I^{\alpha}_j,\ell}^2 \leq 2K\mu_\ell \theta_{I^{\alpha}_j,\ell}.
\end{align*}
Therefore,
\begin{align*}
&\left|\frac{1}{n}\sum_{w=1}^n\prod_{j=1}^{|\alpha|}  \sum_{\ell=1}^\infty \mu_\ell f_\ell(X_w) \mathcal{C}_{I^{\alpha}_j,\ell}- \frac{1}{n}\sum_{w=1}^n\prod_{j=1}^{|\alpha|}  \mu_{I^{\alpha}_j} f_{I^{\alpha}_j}(X_w) \theta_{I^{\alpha}_j}\right|\\
&\leq \frac{1}{n}\sum_{w=1}^n \left|\prod_{j=1}^{|\alpha|} A_{I^{\alpha}_j}(X_w)- \prod_{j=1}^{|\alpha|} B_{I^{\alpha}_j}(X_w)\right|\\
&\leq |\alpha| (2K\mu_\ell \theta_{I^{\alpha}_j,\ell})^{|\alpha|} \sup_{j\in [K]}|A_j-B_j|\\
&\stackrel{(P)}{\longrightarrow} 0,
\end{align*}
as $n$ goes to infinity. This implies that
\begin{align*}
n^{\frac{|\alpha|}{2}-1}\sum_{w=1}^n\prod_{j=1}^{|\alpha|} B_{I^{\alpha}_j}^tQ^X(w)&=\frac{1}{n}\sum_{w=1}^n\prod_{j=1}^{|\alpha|} \sum_{i=1}^n \sqrt{n}B_{I^{\alpha}_j}(i)\left (\sum_{\ell=1}^\infty \mu_\ell f_\ell(X_i)f_\ell(X_w) \right )\\ 
&=\frac{1}{n}\sum_{w=1}^n\prod_{j=1}^{|\alpha|}  \sum_{\ell=1}^\infty \mu_\ell f_\ell(X_w) \mathcal{C}_{I^{\alpha}_j,\ell}\\
&\stackrel{(P)}{\longrightarrow} \frac{1}{n}\sum_{w=1}^n\prod_{j=1}^{|\alpha|}  \mu_{I^{\alpha}_j} f_{I^{\alpha}_j}(X_w) \theta_{I^{\alpha}_j}\\
&= \prod_{i=1}^K (\theta_i \mu_i)^{\alpha_i} \frac{1}{n}\sum_{w=1}^n f_\alpha (X_w),
\end{align*}
as $n$ goes to infinity. We also notice that, by law of large number,
\begin{align*}
    \frac{1}{n}\sum_{w=1}^n f_\alpha (X_w) \rightarrow \int_{[0,1]} f(x) dx, a.s. 
\end{align*}
Therefore, the lemma follows.
\end{proof}

Consequently,
\begin{lem}\label{l:errorcount}
For $0\leq \alpha\leq \mathbf{N}$, we have that
\begin{align*}
    P_\alpha=\frac{A_{\alpha}n^{|\alpha|/2-1}}{\epsilon^{|\alpha|} \prod_{i=1}^K (\sqrt{P_{ii}} \lambda_i)^{\alpha_i}}\stackrel{(P)}{\longrightarrow} \int_{[0,1]} f_\alpha dx,
\end{align*}
as $n$ goes to infinity.
\end{lem}

\subsection{Polynomial Approximation}
Let $U$ be Uniform on $[0,1]$ and let $N_i$ be a series of iid random variables whose distribution function is proportional to $\Psi_\delta(x)$. Define $Y_i=f_i(U)+N_i$. Then the joint distribution of $(Y_1,\cdots,Y_K)$ is the same as $(f_1(U)+N_1,\cdots, f_K(U)+N_K)$. Therefore,
\begin{align*}
    \int f_\alpha dx
    &=\mathbb{E}[f_{1}(U)^{\alpha_1}\cdots f_{K}(U)^{\alpha_K}]=\mathbb{E}[f_\alpha(U)].
\end{align*}
We can also relate this term to the joint moments of the $Y_i$'s.
\begin{align*}
    \mathbb{E}[Y_{1}^{\alpha_1}\cdots Y_{K}^{\alpha_K}]
    &=\mathbb{E}[(f_{1}(U)+N_1)^{\alpha_1}\cdots (f_{K}(U)+N_K)^{\alpha_K}]\\
    &=\sum\limits_{0\leq \beta\leq \alpha}\mathbb{E}[f_\beta(U)]\prod_{i=1}^K {\alpha_i \choose \beta_i}\mathbb{E}[N_\delta^{\alpha_i-\beta_i}]\\
    &=:\mu_\alpha(\delta),
\end{align*}
where $N_\delta$ is a random variable whose distribution function is proportional to $\Psi_\delta(x)$. 

\begin{lem}\label{l:moment1}
If $|P_\alpha-\int_{[0,1]} f_\alpha dx|\leq e_2$ for any $0\leq \alpha\leq \mathbf{N}$, then
\begin{align*}
    \sum_{0\leq \alpha\leq N}|M_\alpha(\delta)- \mu_\alpha(\delta)|^2\leq e_3^2:=e_2^2 \left( \sum_{i=0}^N\mathbb{E}[(1+N_\delta)^i]^2\right)^K.
\end{align*}
\end{lem}
\begin{proof}
Recall that in the function \textproc{Mollifier\_Convolution}, we defined
\begin{align*}
    M_{\alpha}(\delta)=\sum\limits_{0\leq \beta\leq \alpha}P_\beta\prod_{i=1}^K {\alpha_i \choose \beta_i}\mathbb{E}[N^{\alpha_i-\beta_i}].
\end{align*}
So by our assumption,
\begin{align*}
    \sum_{0\leq \alpha\leq N}|M_{\alpha}(\delta)-\mu_\alpha(\delta)|^2
    &\leq \sum_{0\leq \alpha\leq \mathbf{N}} \left(e_2\sum\limits_{0\leq \beta\leq \alpha}\prod_{i=1}^K {\alpha_i \choose \beta_i}\mathbb{E}[N_\delta^{\alpha_i-\beta_i}]\right)^2\\
    &= e_2^2\sum_{0\leq \alpha\leq \mathbf{N}}  \mathbb{E}[(1+N_\delta)^{\alpha_1}]^2\cdots\mathbb{E}[ (1+N_\delta)^{\alpha_K}]^2\\
    &= e_2^2 \left( \sum_{i=0}^N\mathbb{E}[(1+N_\delta)^i]^2\right)^K.
\end{align*}
\end{proof}
We note that being the $i$-th eigenfunction of $Q$, 
\begin{align*}
    |\mu_i f_i(x)|=|\int_0^1 Q(x,y)f_i(y)dy| \leq (\int_0^1 Q(x,y)^2dy)^{1/2}\leq M
\end{align*}
So,
\begin{align*}
    |f_i(x)|\leq \frac{M}{|\mu_i|}\leq \frac{M}{\sqrt{\mu_1}},
\end{align*}
for $i\in [r_0]$. With high probability, as $|\lambda_1-\mu_1|\leq e_1(n)\leq 1\leq \mu_1$, we have that  $\sqrt{\lambda_1} \leq  3\sqrt{\mu_1}/2$. So $|f_i(x)|\leq \frac{3M}{2\sqrt{\lambda_1}}$. Further, as $\delta\leq\frac{1}{2}\leq \frac{M}{2\sqrt{\lambda_1}}$, we have $|Y_i|\leq \frac{2M}{\sqrt{\lambda_1}}=\kappa$ with high probability. Write $u(x_1,\cdots,x_K)$ as the distribution function of $(Y_1,\cdots,Y_K)$. So with high probability, $u$ is identically zero outside the box $[-\kappa,\kappa]^K$. Moreover, $u$ is a smooth function as $\Psi$ is smooth. Now, we use Legendre polynomials in $[-\kappa,\kappa]^K$ to approximate $u$. Write
\begin{align*}
    \rho_\alpha=\int_{[-\kappa,\kappa]^K} u(x_1,\cdots,x_K) \tilde L_\alpha(x_1,\cdots,x_K) d\mathbf{x}.
\end{align*}
Therefore,
\begin{align*}
    \rho_\alpha=\sum\limits_{0\leq \beta\leq \alpha} \tilde C^{\otimes K}_{\beta} \mu_\alpha.
\end{align*}
One way to approximate $u$ is to consider the following term
\begin{align*}
    h_N(x_1,\cdots,x_K):=\sum\limits_{0\leq \beta\leq \mathbf{N}}\rho_\beta \tilde L_\beta(x_1,\cdots,x_K).
\end{align*}
Define
\begin{align*}
    &t_N(x_1,\cdots,x_K):=\sum\limits_{\cup_i \{\beta:\beta_i >N\}}\rho_\beta \tilde L_\beta(x_1,\cdots,x_K),\\
    &t_{i,N}(x_1,\cdots,x_K):=\sum\limits_{\beta:\beta_i >N}\rho_\beta \tilde L_\beta(x_1,\cdots,x_K).
\end{align*}
We define $H_N$ to be a $(N+1)$ by $(N+1)$ matrix such that its $(i,j)$-th entry equals to $\frac{2}{i+j-1}$, when $i+j$ is even and $0$ elsewhere. Write $\mathrm{Cond}(H_N)$ as its condition number.
\begin{lem} \label{l:moment2}
If
\begin{align*}
    \sum_{0\leq \alpha\leq N}|M_\alpha(\delta)- \mu_\alpha(\delta)|^2\leq e_3^2,
\end{align*}
then
\begin{align*}
    \int_{[-\kappa,\kappa]^K}|\hat h_N(x_1,\cdots,x_K)- h_N(x_1,\cdots,x_K)|^2 d\mathbf{x}\leq e_3^2 \kappa^{2N+1} (\mathrm{Cond}(H_N))^K.
\end{align*}
\end{lem}
\begin{proof}
Recall that $\tilde C_{ij}$ are scaled Legendre polynomial coefficients. Define $\tilde C$ to be a $(N+1)$ by $(N+1)$ matrix such that its $(i+1,j+1)$-th entry equals to $\tilde C_{ij}$. Define a matrix $\tilde H_N=(\tilde C^T \tilde C)^{-1}$. Note that the $(i,j)$-th entry of $\tilde H_N$ equals to $\frac{2\kappa^{i+j-1}}{i+j-1}$, when $i+j$ is even and $0$ elsewhere. The reason is as follows. Define $\tilde M=\tilde C^{-1}$. Write $\tilde M_{ij}$ as $(i+1,j+1)$-th entry of $\tilde M$. Then
\begin{align*}
    x^i=\sum_{j=0}^i \tilde M_{ij}\tilde L_j(x).
\end{align*}
Thus
\begin{align*}
    \int_{[-\kappa,\kappa]} x^i \sum_{j=0}^k \tilde C_{kj} x^j dx=\sum_{j=0}^k \tilde C_{kj} \frac{(1-(-1)^{i+j+1})\kappa^{i+j+1}}{i+j+1} =\tilde M_{ik}.
\end{align*}
Therefore, $\lambda_{\min} (\tilde H_N)\geq \kappa^{-2N-1}\lambda_{\min} ( H_N).$ We write $\tilde H_N^{\otimes K}$ as the Kronecker product of $K$ copies of $\tilde H_N$. Write $M(\delta)$ and $\mu(\delta)$ as vector of $M_\alpha(\delta)$ and $\mu_\alpha(\delta)$ respectively. Then 
\begin{align*}
    \int_{[-\kappa,\kappa]^K} |\hat h_N(x_1,\cdots,x_K)- h_N(x_1,\cdots,x_K)|^2 d\mathbf{x}&=\langle(H_N^{\otimes K})^{-1} (M(\delta)- \mu(\delta)),M(\delta)- \mu(\delta)\rangle\\
    &\leq e_3^2 \kappa^{2N+1} (\mathrm{Cond}(H_N))^K.
\end{align*}
The last inequality follows from the fact that the largest eigenvalue of $\lambda_{\max}(H_N)\geq 1$.
\end{proof}

In order to control the tail terms, we need to have the following inequality first.
\begin{lem}
\begin{align*}
    \int_{[-\kappa,\kappa]^K} |\frac{d}{dx_i}u(x)|^2 d\mathbf{x}\leq (2\kappa)^K\delta^{-2K-8}.
\end{align*}
\end{lem}
\begin{proof}

We denote by $T$ the distribution of $(f_1(U),\cdots,f_K(U))$. We consider it as a continuous linear functional on $C^\infty([-\kappa,\kappa]^K)$. We also define $\Psi_\delta^{\otimes K}(x_1,\cdots,x_K)=\Psi_\delta(x_1)\cdots\Psi_\delta(x_K)$. Then
\begin{align*}
    u(x_1,\cdots,x_K)=(T*C_\delta \Psi_\delta^{\otimes K})(x_1,\cdots,x_K),
\end{align*}
where $C_\delta$ is the normalization constant of $\Psi_\delta^{\otimes K}$ such that $\|C_\delta \Psi_\delta^{\otimes K}\|_1=1$. Thus
\begin{align*}
\|\frac{d}{dx_i}u(x)\|_{L_\infty([-\kappa,\kappa]^K)}&=C_\delta\|T*\frac{d}{dx_i}\Psi_\delta^{\otimes K}\|_{L_\infty([-\kappa,\kappa]^K)}\\
&\leq C_\delta \|\frac{d}{dx_i}\Psi_\delta^{\otimes K}\|_{L_\infty([-\kappa,\kappa]^K)}\\
&= C_\delta\|\frac{-2x_i}{(\delta^2-x_i^2)^2}\Psi_\delta^{\otimes K}\|_{L_\infty([-\kappa,\kappa]^K)}\\
&\leq \delta^{-K-4}.
\end{align*}
Therefore,
\begin{align*}
    \int_{[-\kappa,\kappa]^K} |\frac{d}{dx_i}u(x)|^2 d\mathbf{x}\leq (2\kappa)^K\delta^{-2K-8}.
\end{align*}
\end{proof}

\begin{lem} \label{l:moment3}
\begin{align*}
    \int_{[-\kappa,\kappa]^K} |t_N(x_1,\cdots,x_K)|^2d\mathbf{x}\leq \frac{K^22^K\kappa^{K+2}\delta^{-2K-8}}{(N+1)^2}.
\end{align*}
\end{lem}
\begin{proof}
Recall that 
\begin{align*}
    -\frac{d}{dx}[(1-x)(1+x)L_i'(x)]=i(i+1)L_i(x),
\end{align*}
the differential equation for Legendre polynomials. So
\begin{align*}
    -\frac{d}{dx_i}[(1-x_i)(1+x_i)\frac{d}{dx_i}L_{\alpha}(x)]=\alpha_i(\alpha_i+1)L_\alpha(x).
\end{align*}
and
\begin{align*}
    -\frac{d}{dx_i}[(\kappa-x_i)(\kappa+x_i)\frac{d}{dx_i}\tilde L_{\alpha}(x)]=\alpha_i(\alpha_i+1)\tilde L_\alpha(x).
\end{align*}
We multiply both sides by $u$ and do integration by parts twice,
\begin{align*}
    &\int_{[-\kappa,\kappa]^K} \tilde L_\alpha(x)\frac{d}{dx_i}[-(\kappa-x_i)(\kappa+x_i)\frac{d}{dx_i}u(x)] d\mathbf{x}\\
    =&\alpha_i(\alpha_i+1)\int_{[-\kappa,\kappa]^K} \tilde L_\alpha(x) u(x)d\mathbf{x}=\alpha_i(\alpha_i+1) \rho_\alpha.
\end{align*}
By Parseval's theorem and integration by parts,
\begin{align*}
    \sum_{\alpha} \alpha_i(\alpha_i+1)|\rho_\alpha|^2&=\int_{[-\kappa,\kappa]^K} u(x)\frac{d}{dx_i}[-(\kappa-x_i)(\kappa+x_i)\frac{d}{dx_i}u(x)] d\mathbf{x}\\
    &=\int_{[-\kappa,\kappa]^K} (\kappa-x_i)(\kappa+x_i)|\frac{d}{dx_i}u(x)|^2 d\mathbf{x}.
\end{align*}
By Pareval's theorem and the definition of $t_{i,N}$,
\begin{align*}
    \int_{[-\kappa,\kappa]^K} |t_{i,N}|^2 d\mathbf{x}&=\sum_{\alpha:\alpha_i >N}\rho_\alpha^2\\
    &\leq \frac{1}{(N+1)^2} \sum_{\alpha:\alpha_i >N} \alpha_i(\alpha_i+1) \rho_\alpha^2\\
    &\leq \frac{1}{(N+1)^2}\int_{[-\kappa,\kappa]^K} (\kappa-x_i)(\kappa+x_i)|\frac{d}{dx_i}u(x)|^2 d\mathbf{x}\\
    &\leq \frac{2^K\kappa^{K+2}\delta^{-2K-8}}{(N+1)^2}.
\end{align*}
This implies that
\begin{align*}
    \int_{[-\kappa,\kappa]^K} |t_N(x_1,\cdots,x_K)|^2d\mathbf{x}\leq K\sum_{i=1}^K\int_{[-\kappa,\kappa]^K} |t_{i,N}(x_1,\cdots,x_K)|^2d\mathbf{x}\leq \frac{K^22^K\kappa^{K+2}\delta^{-2K-8}}{(N+1)^2}.
\end{align*}
\end{proof}
Combining Lemma \ref{l:errorcount}, Lemma \ref{l:moment1}, Lemma \ref{l:moment2}, Lemma \ref{l:moment3} and our definition of $N$ and $\delta$, we have that with high probability,
\begin{align*}
    \int_{[-\kappa,\kappa]^K} |\hat h_N-u|^2 d\mathbf{x}&\leq e_2^2 \left( \sum_{i=0}^N\mathbb{E}[(1+N_\delta)^i]^2\right)^K \kappa^{2N+1} (\mathrm{Cond}(H_N))^K+\frac{K^22^K\kappa^{K+2}\delta^{-2K-8}}{(N+1)^2}\\
    &\leq 2^{-K-14} e_0^2 K^{-2} \lambda_1^{-4} \kappa^{-K-8},
\end{align*}
as $n$ goes to infinity. Therefore,
\begin{align*}
    \int_{[-\kappa,\kappa]^K} |\hat h_N-u| d\mathbf{x} \leq \frac{e_0}{2^7K\lambda_1^2 \kappa^4}.
\end{align*}
Notice that this also implies that
\begin{align*}
    \int_{[-\kappa,\kappa]^K} |\hat h_N^+-\hat h_N| d\mathbf{x} \leq \frac{e_0}{2^7K\lambda_1^2 \kappa^4}.
\end{align*}
Thus
\begin{align*}
    \int_{[-\kappa,\kappa]^K} |\hat h_N^+-u| d\mathbf{x} \leq \frac{e_0}{2^6K\lambda_1^2 \kappa^4}.
\end{align*}
Further,
\begin{align*}
    \int_{[-\kappa,\kappa]^K} |\frac{\hat h_N^+}{\|\hat h_N^+\|_1}-u| d\mathbf{x} \leq \frac{e_0}{2^5K\lambda_1^2 \kappa^4},
\end{align*}
as $n$ goes to infinity with high probability.

\subsection{Graphon Estimation}
Recall that in \textproc{Graphon\_Estimation}, we sampled $Z_1,\cdots,Z_m$ according $\hat h_N^+$ mutually independently. For any $1\leq i\leq K$, we defined 
\begin{align*}
    \hat f_i(x)=Z_{\lceil xm\rceil }(i).
\end{align*}
and set
\begin{align*}
    \hat Q(x,y)=\sum\limits_{i=1}^K \lambda_i \hat f_i(x)\hat f_i(y).
\end{align*}

\begin{lem}
Let $U_1,\cdots,U_m$ be iid random variables following Uniform $[0,1]$ distribution. Then there exists a measure preserving map $\phi:[0,1]\rightarrow [0,1]$, such that
\begin{align*}
    \int_{[0,1]}|f_k(U_{\lceil \phi(x)m\rceil })-f_k(x)|^2 dx \stackrel{(P)}{\longrightarrow} 0,
\end{align*}
as $m$ goes to infinity.
\end{lem}
\begin{proof}
Define a permutation $\sigma \in S_m$ such that $U_{\sigma(1)}\leq U_{\sigma(2)}\leq \cdots \leq U_{\sigma(m)}$. For $i\in [m]$, define
\begin{align*}
    \phi(x)=\frac{\sigma(i)-i}{m}+x, \ \ x\in (\frac{i-1}{m},\frac{i}{m}].
\end{align*}
Then $U_{\lceil \phi(x)m\rceil }$ is a piecewise constant monotone function in $(0,1]$. For any $x\in (0,1]$, we can write down the distribution of $U_{\lceil \phi(x)m\rceil }$ explicitly. For $\frac{i-1}{m}< x\leq \frac{i}{m}$, $U_{\lceil \phi(x)m\rceil }$ follows Beta $(i,m-i+1)$ distribution. We will show that, for any $\alpha>0$,
\begin{align*}
    \mathbb{P}[\int_{[0,1]} |f_k(U_{\lceil \phi(x)m\rceil })-f_k(x)|^2 dx >\alpha] \rightarrow 0.
\end{align*}
As $f_k$ is a bounded function, it is in $L_2([0,1])$. Fix $0<\epsilon<\alpha/3$. Using the density of continuous functions in $L_2([0,1])$, we find a continuous function $g:[0,1]\rightarrow \mathbb{R}$ such that $\|f_k-g\|_2<\epsilon$. Thus
\begin{align*}
    &\mathbb{P}[\int_{[0,1]} |f_k(U_{\lceil \phi(x)m\rceil })-f_k(x)|^2 dx >\alpha]\\
    =& \mathbb{P}[\sum_{i=1}^m\int_{\frac{i-1}{m}}^{\frac{i}{m}} |f_k(U_{\lceil \phi(x)m\rceil })-f_k(x)|^2dx>\alpha].
\end{align*}    
Notice that 
\begin{align*}
\sum_{i=1}^m\int_{\frac{i-1}{m}}^{\frac{i}{m}} |f_k(U_{\lceil \phi(x)m\rceil })-f_k(x)|^2dx &\leq 
 \frac{1}{m}\sum_{i=1}^m |f_k(U_{\sigma(i)})-g(U_{\sigma(i)})|^2dx\\
&+\sum_{i=1}^m\int_{\frac{i-1}{m}}^{\frac{i}{m}} |g(U_{\lceil \phi(x)m\rceil })-g(x)|^2dx+
\int_{0}^{1} |g(x)-f_k(x)|^2dx
\end{align*}
In order for the left hand side to be larger than $\alpha$, at least one of terms on the right hand side should be larger than $\alpha/3$. The third term is not possible by definition. So we have
\begin{align*}
    &\mathbb{P}[\sum_{i=1}^m\int_{\frac{i-1}{m}}^{\frac{i}{m}} |f_k(U_{\lceil \phi(x)m\rceil })-f_k(x)|^2dx>\alpha]\\
    \leq & \mathbb{P}[\frac{1}{m}\sum_{i=1}^m |f_k(U_{\sigma(i)})-g(U_{\sigma(i)})|^2dx>\alpha/3]+
    \mathbb{P}[\sum_{i=1}^m\int_{\frac{i-1}{m}}^{\frac{i}{m}} |g(U_{\lceil \phi(x)m\rceil })-g(x)|^2dx>\alpha/3].
\end{align*}
The first term equals 
\begin{align*}
    \mathbb{P}[\frac{1}{m}\sum_{i=1}^m |f_k(U_i)-g(U_i)|^2dx>\alpha]\leq \frac{\epsilon^2}{\alpha}.
\end{align*}
To control the second term, consider the event
$A_i:=\{|U_{\lceil \phi(x)m\rceil }-\frac{i}{m}|\leq m^{-\frac{1}{4}}\}$. Then
\begin{align*}
    \mathbb{P}[U_{\lceil \phi(x)m\rceil }-\frac{i}{m} \geq m^{-\frac{1}{4}}]=\mathbb{P}[Y_i\leq i]\leq e^{-\frac{1}{4}\sqrt{m}},
\end{align*}
by Chernoff bound, where $Y_i$ follows Binomial$(m,\frac{i}{m}+m^{-\frac{1}{4}})$ distribution. Similarly, when $\frac{i}{m} - m^{-\frac{1}{4}}\geq 0$,
\begin{align*}
    \mathbb{P}[U_{\lceil \phi(x)m\rceil }\leq \frac{i}{m} - m^{-\frac{1}{4}}]=\mathbb{P}[Y_i\geq i]\leq e^{-\frac{1}{2}\sqrt{m}},
\end{align*}
by Chernoff bound, where $Y_i$ follows Binomial$(m,\frac{i}{m}-m^{-\frac{1}{4}})$ distribution. Therefore, if we define $A:=\{\forall i: |U_{\lceil \phi(x)m\rceil }-\frac{i}{m}|\leq m^{-\frac{1}{4}}\}$, then $\mathbb{P}[A]\geq 1-2me^{-\frac{1}{4}\sqrt{m}}$. Note that there exists $\delta$, such that $|g(x)-g(y)|< \alpha/3$, whenever $|x-y|<\delta$. So when $m$ is large,
\begin{align*}
    \mathbb{P}[\sum_{i=1}^m\int_{\frac{i-1}{m}}^{\frac{i}{m}} |g(U_{\lceil \phi(x)m\rceil })-g(x)|^2dx>\alpha/3] \leq \mathbb{P}[A^c]\leq 2me^{-\frac{1}{4}\sqrt{m}}.
\end{align*}
Therefore, by taking arbitrarily small $\epsilon$, we have that for any $\alpha>0$,
\begin{align*}
    \mathbb{P}[\int_{[0,1]} |f_k(U_{\lceil \phi(x)m\rceil })-f_k(x)|^2 dx >\alpha] \rightarrow 0.
\end{align*}
\end{proof}

\begin{lem}\label{l:l2error}
Let $Z_1,\cdots,Z_m$ be iid random vectors following the distribution proportional to $\hat h_N^+$. Then there exists a measure preserving map $\phi:[0,1]\rightarrow [0,1]$, such that 
\begin{align*}
   \int_{[0,1]}|Z_{\lceil \phi(x)m\rceil }(i)-f_i(x)|^2 dx \leq \frac{e_0}{4K\lambda_1^2 \kappa^2}
\end{align*}
with high probability.
\end{lem}
\begin{proof}
Let $U_1,\cdots,U_m$ be iid random variables following Uniform $[0,1]$ distribution. Define $W_i=(f_1(U_i),\cdots,f_K(U_i))$. Define $Y_i=(f_1(U_i)+N_{1,i},\cdots,f_K(U_i)+N_{K,i})$, where $N_{j,i}$ follows a distribution proportional to $\Psi_\delta$ jointly independently. We consider a coupling of $(Z_1,\cdots,Z_m)$ and $(Y_1,\cdots,Y_m)$ such that $Z_i\neq Y_i$ with probability $1-\|\frac{\hat h_N^+}{\|\hat h_N^+\|_1} -u\|_{TV} $. Take $\phi$ as in the previous lemma. Then 
\begin{align*}
    &\int_{[0,1]}|Z_{\lceil \phi(x)m\rceil }(i)-f_i(x)|^2 dx\\
    \leq&  \int_{[0,1]}|W_{\lceil \phi(x)m\rceil }(i)-f_i(x)|^2 dx+  \int_{[0,1]}|Y_{\lceil \phi(x)m\rceil }(i)-W_{\lceil \phi(x)m\rceil }(i)|^2 dx\\
    &+\int_{[0,1]}|Z_{\lceil \phi(x)m\rceil }(i)-Y_{\lceil \phi(x)m\rceil }(i)|^2 dx\\
    =&  \int_{[0,1]}|W_{\lceil \phi(x)m\rceil }(i)-f_i(x)|^2 dx+  \frac{1}{m}\sum_{j=1}^m |N_{i,j}|^2+\frac{1}{m}\sum_{j=1}^m |Z_j(i)-Y_j(i)|^2.
\end{align*}
The first term is controlled by the previous lemma. For the second term, note that each term in the summation is independent and each term satisfies $|N_{i,j}|^2 \leq \delta^2\leq \frac{e_0}{2^4 K \lambda_1^2 \kappa^2}$. For the third term, note that with high probability,
\begin{align*}
    \mathbb{E}[ |Z_j(i)-Y_j(i)|^2]\leq 4\kappa^2 \|\frac{\hat h_N^+}{\|\hat h_N^+\|_1} -u\|_{TV}\leq 4\kappa^2 \frac{e_0 }{2^{5}K\lambda_1^2 \kappa^4}= \frac{e_0}{2^3K\lambda_1^2 \kappa^2}.
\end{align*}
All together, the statement follows.
\end{proof}

\begin{lem}
There is a measure preserving map $\phi:[0,1]\rightarrow [0,1]$, such that
\begin{align*}
\int_{[0,1]^2}|\sum\limits_{i=1}^K \mu_i f_i(x) f_i(y)-\sum\limits_{i=1}^K \lambda_i \hat f_i(\phi(x)) \hat f_i(\phi(y))|^2 dxdy < e_0,
\end{align*}
with high probability as n goes to infinity.
\end{lem}
\begin{proof}
We separate the left hand side into two terms,
\begin{align*}
&\int_{[0,1]^2}|\sum\limits_{i=1}^K \mu_i f_i(x) f_i(y)-\sum\limits_{i=1}^K \lambda_i \hat f_i(\phi(x)) \hat f_i(\phi(y))|^2 dxdy\\
\leq & \int_{[0,1]^2}|\sum\limits_{i=1}^K \mu_i f_i(x) f_i(y)-\sum\limits_{i=1}^K \lambda_i f_i(x) f_i(y)|^2 dxdy\\
+&\sum_{i=1}^K \int_{[0,1]^2}|\lambda_i f_i(x) f_i(y)-\lambda_i \hat f_i(\phi(x)) \hat f_i(\phi(y))|^2 dxdy.
\end{align*}
Note that with high probability the first term is bounded by $O(\frac{1}{\log(n)})$. With high probability, the second term is bounded by
\begin{align*}
&\sum_{i=1}^K \lambda_i^2\int_{[0,1]}| f_i(x)|^2 dx  \int_{[0,1]}|f_i(y)-\hat f_i(\phi(y))|^2 dy\\
+&\sum_{i=1}^K \lambda_i^2\int_{[0,1]}| \hat f_i(\phi(y))|^2 dy  \int_{[0,1]}|f_i(x)-\hat f_i(\phi(x))|^2 dx\\
\leq&  2K \lambda_1^2\kappa^2\sup_{i\in [K]}\int_{[0,1]}|f_i(x)-\hat f_i(\phi(x))|^2 dx\\
\leq&  \frac{e_0}{2}.
\end{align*}
Therefore, the statement follows.
\end{proof}

\subsection{Proof of the Corollary}
We recall the corollary: Let $Q$ satisfy Assumptions $1$ and $2$ and assume that all nonzero $\mu_i$'s are simple, then Algorithm $1$ produces an estimator $\widehat{\mathcal{Q}_h}$ such that
\begin{align*}
    \delta_2(\frac{1}{h}\widehat{\mathcal{Q}_h}, Q)\stackrel{(P)}{\longrightarrow} o_h(1).
\end{align*}
\begin{proof}
Note that $\|Q_k-Q\|_2\rightarrow 0$ as $k$ goes to infinity. For any $\epsilon>0$, take $k_0$ large enough such that $\|Q_k-Q\|_2<\epsilon$ for any $k\geq k_0$. We also know that $|\mu_i|\rightarrow 0$ as $i$ goes to infinity. Take $H> \max\{1, \frac{\mu_1}{\mu_{k_0}^2}\}$. Then for any $h\geq H$, the top $k_0$ eigenvalues of $\mathcal{Q}_h$ have magnitudes larger than the square root of the top eigenvalue. This implies that
\begin{align*}
    \delta_2(\frac{1}{h}\widehat{\mathcal{Q}_h}, Q_{k_0})\stackrel{(P)}{\longrightarrow} 0,
\end{align*}
as $n$ goes to infinity. We also note that
\begin{align*}
   \delta_2(\frac{1}{h}\widehat{\mathcal{Q}_h}, Q) &\leq
   \delta_2(\frac{1}{h}\widehat{\mathcal{Q}_h}, Q_{k_0})+\|Q_{k_0}-Q\|_2\\
   &< \delta_2(\frac{1}{h}\widehat{\mathcal{Q}_h}, Q_{k_0})+\epsilon.
\end{align*}
Therefore, the statement follows.
\end{proof}

\section{Proof of Lemma \ref{l:cconverge}}
To prove the convergence of $\mathcal{C}_{ij}$, we need to use results in \cite{bordenave2018nonbacktracking} and generalized versions for graphons. In the following subsection, we cite necessary definitions and results.

\subsection{Some Results from \cite{bordenave2018nonbacktracking}}
In \cite{bordenave2018nonbacktracking}, the model is for constant expected degree stochastic block model. Here are some notations that the authors use. Let $\pi=(\pi(1),\cdots,\pi(r))$ be a probability vector. Each vertex $v\in [n]$ is given a type $\sigma(v)\in [r]$ independently according to $\pi$. The probability that two vertices $u$ and $v$ are connected is $W(\sigma(u),\sigma(v))/n$. Set $\Pi=diag(\pi(1),\cdots,\pi(r))$ 
independently. Define $M=\Pi W$. Denote $\alpha=\sum_{i=1}^r \pi(i)W_{ij}=\sum_{i=1}^r M_{ij}$. Let $\phi_i$ be normalized left eigenvectors of $M$. 
\begin{defn}
We consider a multi-type branching process where a particle of type $j \in [r]$ has a $\mathrm{Poi}(M_{ij})$ number of children with type $i$. We
denote by $Z_t=(Z_t(1),\dots, Z_t(r))$ the population at generation
$t$, where $Z_t(i)$ is the number of particles at generation $t$ with
type $i$. 
\end{defn}

\begin{defn}
We consider a functional of the multi-type branching process which depends on particles in more than one generation. More precisely, assuming that $\|Z_0 \|_1= 1$, we denote by $V$ the particles of the random tree and $o \in V$ the starting particle. Particle $v \in V$ has type $\sigma(v) \in [r]$ and generation $|v|$ from $o \in V$. For $v \in V$ and integer $t \geq 0$, let $Y^v_t$ denote the set of particles of generation $t$ from $v$ in  the subtree of particles with common ancestor $v \in V$. Finally, $Z^v_t  = ( Z^v_t (1), \cdots , Z^v _t (r))$ is the vector of population at generation $t$ from $v$, i.e. $Z_t ^v (i) = \sum_{u \in Y^v_t} \mathbbm{1}(  \sigma(u) = i)$. We set 
\begin{align*}
S^v _t = \| Z^v _t \|_1 = \langle \phi_1 , Z^v _t \rangle.
\end{align*}
We fix an integer $ k \in [r]$, $\ell \geq 1$ and set 
\begin{equation*}
Q_{k,\ell} = \sum_{(u_0,\ldots,u_{2\ell+1})\in{\mathcal P}_{2\ell+1}} \phi_k ( \sigma (u_{2 \ell +1}) ) ,
\end{equation*}
where the sum is over $ (u_0, \ldots, u_{2 \ell +1})\in{\mathcal P}_{2\ell+1}$, the set of paths in the tree starting from $u_0 = o$ of length $2 \ell +1$ with $(u_0, \ldots,
u_{\ell})$ and $(u_{\ell}, \ldots, u_{2 \ell +1})$ non-backtracking and
$u_{\ell -1}  = u_{\ell +1}$ (i.e. $ (u_0, \ldots, u_{2 \ell +1})$
backtracks exactly once at the $\ell+1$-th step).
\end{defn}

The following alternative representation of $Q_{k,\ell}$ will prove useful. By distinguishing paths $ (u_0, \ldots, u_{2 \ell +1})$ according to the smallest depth $t\in\{0,\ldots,\ell-1\}$ to which they climb back after visiting $u_{\ell+1}$ and the node $u_{2\ell -t}$ they then visit at level $t$ we have that
\begin{equation*}
Q_{k,\ell} =  \sum_{t=0} ^{\ell -1} \sum_{ u \in Y^o_{t}} L_{k,\ell}^u,
\end{equation*}
where we let for $|u|=t\geq 0$,
\begin{align*}
    L^u_{k,\ell} = \sum_{w\in Y^u_1} S^w_{\ell-t-1}\left( \sum_{v\in
    Y^u_1\backslash \{w\}} \langle \phi_k, Z^v_t\rangle \right).
\end{align*}

\begin{defn}
Let $Z_t$, $t\geq 0$, be the Galton-Watson branching process defined above started from $Z_0 = \delta_{\iota}$. We denote by $(T,o)$ the associated random rooted tree.  Let $D$ be the number of offspring of the root and for $1 \leq x \leq D$,  let $Q_{k,\ell}(x) $ be the random variable $Q_{k,\ell}$ defined on the tree $T^x$ where the subtree attached to $x$ is removed and set 
\begin{align*}
    J_{k,\ell} = \sum_{x=1}^D Q_{k,\ell} (x).
\end{align*}
\end{defn}

\begin{lem}[Lemma 41 in \cite{bordenave2018nonbacktracking}]
We have $J_{k,\ell}   \mu_k ^{-2\ell}  - \alpha  \mu_k\phi_k (\iota)/(\mu_k ^2 / \alpha -1)$ converges in $L^2$ to a centered variable $Y_{k}$ satisfying $\mathbb{E}  |Y_{k}| \leq C$.
\end{lem}
This lemma is proved by applying the following theorem.

\begin{thm}[Theorem 25 in \cite{bordenave2018nonbacktracking}]
Assume $Z_0 = \delta_x$.  For $k \in [r_0]$, $Q_{k,\ell}/ \mu_k ^{2 \ell}$ converges in $L^2$ as $\ell$ tends to infinity to a random variable with mean $ \mu_k\phi_k
(x)  / (\mu_k^2 / \alpha-1)$.
For $k \in [r] \backslash[r_0]$, there exists  a constant $C$ such that $\mathbb{E} Q_{k,\ell}^2 \leq  C \alpha ^{2 \ell} \ell^5$. 
\end{thm}

\begin{defn}
For $ e , f \in \vec E(V)$, we define the "oriented" distance 
\begin{align*}
    \vec d ( e, f) = \min_{\gamma} \ell( \gamma )
\end{align*}
where  the minimum is taken over all self-avoiding paths $\gamma = (\gamma_0, \gamma_1, \cdots , \gamma_{\ell+1} )$  in $G$ such that $(\gamma_0, \gamma_1) = e$, $(\gamma_{\ell} , \gamma_{\ell+1} ) = f$ and for all $1 \leq k \leq \ell+1$, $\{ \gamma_k , \gamma_{k+1} \} \in E$ (we do not require that $e \in \vec E$). Observe that $\vec d$ is not symmetric, we have instead $\vec d (e, f) = \vec d (f^{-1}, e^{-1})$. 

Then,  for integer $t \geq 0$,  we introduce the vector $Y_t(e) = (Y_t(e) (i))_{i \in [r]}$ where, for $i \in [r]$,
\begin{equation*}
Y_t (e)(i) = \ABS{ \left\{ f \in \vec E : \vec d ( e, f) = t  , \sigma(f_2) =  i \right\} }.
\end{equation*}
We also set 
\begin{align*}
    S_t (e) = \|Y_t(e) \|_1 = \ABS{ \left\{ f \in \vec E : \vec d ( e, f) = t \right\} }.
\end{align*}
The vector $Y_t(e)$ counts the types at oriented distance $t$ from $e$.  
\end{defn}

\begin{defn}
For $e \in \vec E(V)$, we define for $t \geq 0$, $\mathcal{Y}_t (e) = \{ f \in \vec E : \vec d (e,f) = t \}$. For $k \in [r]$, we set
\begin{equation*}
P_{k,\ell} (e) = \sum_{t = 0} ^{\ell -1} \sum_{f \in \mathcal{Y}_{t} (e) } L_k(f).
\end{equation*}
where 
\begin{align*}
    L_k(f) = \sum_{(g,h) \in \mathcal{Y}_1 (f) \backslash \mathcal{Y}_t(e); g \ne h  } \langle \phi_k , \tilde Y_t(g)\rangle \tilde S_{\ell - t -1} (h),
\end{align*}
and $\tilde Y_t(g)$, $\tilde S_{\ell - t -1}(h) = \|\tilde Y_{\ell - t-1}(h)\|_1$ are the variables $ Y_t(g)$, $S_{\ell - t -1}(h)$ defined on the graph $G$ where all  edges in $(G,e_2)_t$ have been removed. In particular, if $(G,e)_{2\ell}$ is a tree, $\tilde Y_s(g)$ and $Y_s(g)$ coincide for $s \leq 2 \ell -t$.
\end{defn}

\begin{defn}
We introduce a new random variable, for $v \in V$, 
\begin{align*}
   I_{k,\ell} (v) = \sum_{e \in \vec E : e_2 = v} P_{k,\ell} (e), 
\end{align*}
where $P_{k,\ell}$ was defined above.
\end{defn}

\begin{prop}[Proposition 36 in \cite{bordenave2018nonbacktracking}]
Let $\ell \sim \kappa \log_\alpha n$ with $0 < \kappa < 1/2$. There exists $c  >0$, such that if $\tau, \varphi : \mathcal{G}^* \to \mathbb{R}$ are $\ell$-local, $| \tau(G,o) | \leq  \varphi (G,o)$  and $\varphi$ is non-decreasing by the addition of edges, then if $\mathbb{E} \varphi ( T,o)$ is finite, 
\begin{align*}
\mathbb{E}\ABS{\frac 1 n   \sum_{v=1}^n    \tau ( G, v)  - \mathbb{E} \tau ( T, o)  }  & \leq & c   \frac{\alpha^{\ell/2} \sqrt{ \log n}  }{n ^{\gamma/2}}   \PAR{ \PAR{ \mathbb{E} \max_{v \in [n]} \varphi^4 (G,v) } ^{1/ 4 } \vee \PAR{ \mathbb{E} \varphi^2  (T,o)}^{1/2}  },
\end{align*}
where $(T,o)$ is the random rooted tree associated to the Galton-Watson branching process defined previously started from $Z_0 = \delta_{\iota}$ and $\iota$ has distribution $(\pi(1), \ldots, \pi(r))$. Here $0<\gamma<\frac{1}{2}$.
\end{prop}

Below are two lemmas to control the growth of $S$. Recall that $S_t$ is a Galton-Walton branching process with offspring distribution $\mathrm{Poi}(\mu_1)$.
\begin{lem}[Lemma 23 in \cite{bordenave2018nonbacktracking}]
Assume $S_0=1$. There exist $c_0,c_1>0$, such that for all $s\geq 0$,
\begin{align*}
    \mathbb{P}(\forall k\geq 1, S_k\leq s\mu_1^k)\geq 1-c_1e^{-c_0s}.
\end{align*}
\end{lem}

We shall denote by $S_t(v)$ the set of vertices at distance $t$ from $v$.
We introduce 
\begin{align*}
& n(i)  =   \sum_{v=1} ^n \mathbbm{1} ( \sigma(v)  = i) ,
 \quad \quad \pi_n(i)   = \frac { n(i)  }{ n} , \nonumber \\
& \alpha_n (i )  =  \sum_{j=1}^r \pi_n(i) W_{ij}  ,  \quad \quad 
  \bar \alpha_n  = \max_{i \in [r]} \alpha_n (i) = \alpha + O ( n^{-\gamma}),
\end{align*}

\begin{lem}[Lemma 29 in \cite{bordenave2018nonbacktracking}]
There exist $c_0, c_1 >0$ such that for all $s \geq 0$ and for any $w \in [n] \cup \vec E(V)$,    
\begin{align*}
    \mathbb{P} \left( \forall t \geq 0: S_{t} (w)   \leq  s \bar \alpha_n^t  \right)  \geq 1 -  c_1 e^{ - c_0 s}. 
\end{align*}
Consequently, for any $p\geq 1$, there exists $c >0$ such that 
\begin{align*}
    \mathbb{E} \max_{v \in [n] , t \geq 0} \left(  \frac{ S_{t} (v) }{ \bar \alpha^t_n} \right )^p \leq c (\log n)^p.
\end{align*}
\end{lem}

\begin{prop}[Proposition 38 in \cite{bordenave2018nonbacktracking}]
Let $\ell \sim   \kappa \log_\alpha n$ with $0 < \kappa < \gamma/5$. For any $ k\in [r_0]$, there exists $\rho'_k >0$ such that w.h.p.
\begin{align*}
    \frac 1 {\alpha n} \sum_{e \in \vec E}  \frac{P^2_{k,\ell} (e) }{ \mu_k ^{4 \ell}}  \to  \rho'_k.
\end{align*}
\end{prop}

\begin{defn}
We set 
\begin{align*}
    I_{k} (v) = \sum_{e : e_2 = v}  s \sqrt { n} \xi_k' (e),
\end{align*}
where $s = \sqrt{ \alpha \rho'_k}$ and $\rho'_k$ was defined above and $\xi'$ is an orientation of $\xi$.
\end{defn}

\begin{lem}[Lemma 42 in \cite{bordenave2018nonbacktracking}]
We have that w.h.p.
\begin{align*}
    \frac{1} n \sum_{v=1}^n  \ABS{  I_k (v)   -  \frac{I_{k,\ell} (v)}{\mu_k ^ {2 \ell} }}  = o(1).  
\end{align*}
\end{lem}

\subsection{Corresponding Results for Graphon}
This subsection contains corresponding definitions and results for graphons. The proofs are similar as in \cite{bordenave2018nonbacktracking}, so we omit it here. For a vertex $v$, we use $X_v$ or $X(v)\in [0,1]$ to denote its type. We denote $\alpha=\mu_1$.
\begin{defn}
We consider a branching process where a particle of type $x \in [0,1]$ has $\mathrm{Poi}(\mu_1)$ number of children, where their types follow the distribution $Q(x,\cdot)/\mu_1$ independently. 
\end{defn}

\begin{defn}
We consider a functional of the branching process which depends on particles in more than one generation. We denote by $V$ the particles of the random tree and $o \in V$ the starting particle. Particle $v \in V$ has type $X_v \in [0,1]$ and generation $|v|$ from $o \in V$. For $v \in V$ and integer $t \geq 0$, let $Y^v_t$ denote the set of particles of generation $t$ from $v$ in  the subtree of particles with common ancestor $v \in V$. We set $S^v _t = |Y_t^v|$. We fix an integer $ k \in [r_0]$, $\ell \geq 1$ and set 
\begin{equation*}
Q_{k,\ell} = \sum_{(u_0,\ldots,u_{2\ell+1})\in{\mathcal P}_{2\ell+1}} f_k ( X (u_{2 \ell +1}) ) ,
\end{equation*}
where the sum is over $ (u_0, \ldots, u_{2 \ell +1})\in{\mathcal P}_{2\ell+1}$, the set of paths in the tree starting from $u_0 = o$ of length $2 \ell +1$ with $(u_0, \ldots,
u_{\ell})$ and $(u_{\ell}, \ldots, u_{2 \ell +1})$ non-backtracking and
$u_{\ell -1}  = u_{\ell +1}$ (i.e. $ (u_0, \ldots, u_{2 \ell +1})$
backtracks exactly once at the $\ell+1$-th step). 
\end{defn}

The following alternative representation of $Q_{k,\ell}$ will prove useful. By distinguishing paths $ (u_0, \ldots, u_{2 \ell +1})$ according to the smallest depth $t\in\{0,\ldots,\ell-1\}$ to which they climb back after visiting $u_{\ell+1}$ and the node $u_{2\ell -t}$ they then visit at level $t$ we have that
\begin{equation*}
Q_{k,\ell} =  \sum_{t=0} ^{\ell -1} \sum_{ u \in Y^o_{t}} L_{k,\ell}^u,
\end{equation*}
where we let for $|u|=t\geq 0$,
\begin{align*}
    L^u_{k,\ell} = \sum_{w\in Y^u_1} S^w_{\ell-t-1}\left( \sum_{v\in
    Y^u_1\backslash \{w\}} \sum_{w'\in Y_t^v} f_k(X_{w'}) \right).
\end{align*}

\begin{defn}
Consider the branching process defined above started from $X(o)=x$. We denote by $(T,o)$ the associated random rooted tree.  Let $D$ be the number of offspring of the root and for $1 \leq y \leq D$,  let $Q_{k,\ell}(y) $ be the random variable $Q_{k,\ell}$ defined on the tree $T^y$ where the subtree attached to $y$ is removed and set 
\begin{align*}
    J_{k,\ell} = \sum_{x=1}^D Q_{k,\ell} (y).
\end{align*}
\end{defn}

\begin{lem}\label{l:weakconverge}
Consider the branching process defined above started from $X(o)=x$.  For $k \in [r_0]$, we have $J_{k,\ell}   \mu_k ^{-2\ell}  - \alpha  \mu_k f_k (x)/(\mu_k ^2 / \alpha -1)$ converges in $L^2$ to a centered variable $Y_{k,x}$ satisfying $\mathbb{E}  |Y_{k,x}| \leq C$ uniformly for $x\in [0,1]$.
\end{lem}
This lemma is proved by applying the following theorem.

\begin{thm}
Consider the branching process defined above started from $X(o)=x$.  For $k \in [r_0]$, $Q_{k,\ell}/ \mu_k ^{2 \ell}$ converges in $L^2$ as $\ell$ tends to infinity to a random variable with mean $ \mu_k\phi_k
(x)  / (\mu_k^2 / \alpha-1)$.
\end{thm}

\begin{defn}
For $ e , f \in \vec E(V)$, we define the "oriented" distance 
\begin{align*}
    \vec d ( e, f) = \min_{\gamma} \ell( \gamma )
\end{align*}
where  the minimum is taken over all self-avoiding paths $\gamma = (\gamma_0, \gamma_1, \cdots , \gamma_{\ell+1} )$  in $G$ such that $(\gamma_0, \gamma_1) = e$, $(\gamma_{\ell} , \gamma_{\ell+1} ) = f$ and for all $1 \leq k \leq \ell+1$, $\{ \gamma_k , \gamma_{k+1} \} \in E$ (we do not require that $e \in \vec E$). Observe that $\vec d$ is not symmetric, we have instead $\vec d (e, f) = \vec d (f^{-1}, e^{-1})$. 

Then,  for integer $t \geq 0$, we set
\begin{align*}
    S_t (e) = \ABS{ \left\{ f \in \vec E : \vec d ( e, f) = t \right\} }.
\end{align*}
\end{defn}

\begin{defn}
For $e \in \vec E(V)$, we define for $t \geq 0$, $\mathcal{Y}_t (e) = \{ f \in \vec E : \vec d (e,f) = t \}$. For $k \in [r_0]$, we set
\begin{equation}\label{eq:defPkl}
P_{k,\ell} (e) = \sum_{t = 0} ^{\ell -1} \sum_{e' \in \mathcal{Y}_{t} (e) } L_k(e').
\end{equation}
where
\begin{align*}
    L_k(e') = \sum_{(g,h) \in \mathcal{Y}_1 (e') \backslash \mathcal{Y}_t(e); g \ne h  } \sum_{e''\in \tilde{\mathcal{Y}}_t(g)} f_k(X_{e''_2}) \tilde S_{\ell - t -1} (h),
\end{align*}
and $\tilde{\mathcal{Y}}_t(g)$, $\tilde S_{\ell - t -1}(h)$ are the variables $ \mathcal{Y}_t(g)$, $S_{\ell - t -1}(h)$ defined on the graph $G$ where all  edges in $(G,e_2)_t$ have been removed. 
\end{defn}

\begin{defn}
We introduce a new random variable, for $v \in V$, 
\begin{align*}
    I_{k,\ell} (v) = \sum_{e \in \vec E : e_2 = v} P_{k,\ell} (e),
\end{align*}
where $P_{k,\ell}$ was defined above.
\end{defn}

\begin{prop}\label{p:indep}
Let $\ell \sim \kappa \log_\alpha n$ with $0 < \kappa < 1/2$. There exists $c  >0$, such that if $\tau, \varphi : \mathcal{G}^* \to \mathbb{R}$ are $\ell$-local, $| \tau(G,o) | \leq  \varphi (G,o)$  and $\varphi$ is non-decreasing by the addition of edges, then if $\mathbb{E} \varphi ( T,o)$ is finite, 
\begin{eqnarray*}
\mathbb{E}\ABS{\frac 1 n   \sum_{v=1}^n    \tau ( G, v)  - \mathbb{E} \tau ( T, o)  }  & \leq & c   \frac{\alpha^{\ell/2} \sqrt{ \log n}  }{n ^{\gamma/2}}   \PAR{ \PAR{ \mathbb{E} \max_{v \in [n]} \varphi^4 (G,v) } ^{1/ 4 } \vee \PAR{ \mathbb{E} \varphi^2  (T,o)}^{1/2}  },
\end{eqnarray*}
where $(T,o)$ is the random rooted tree associated to the Galton-Watson branching process defined previously started from $X(o) = \iota$ and $\iota$ has distribution Uniform $[0,1]$. Here $0<\gamma<\frac{1}{2}$.
\end{prop}

Recall that $S_t$ is a Galton-Walton branching process with offspring distribution $\mathrm{Poi}(\mu_1)$.
\begin{lem}\label{l:Scontrol1}
Assume $S_0=1$. There exist $c_0,c_1>0$, such that for all $s\geq 0$,
\begin{align*}
    \mathbb{P}(\forall k\geq 1, S_k\leq s\mu_1^k)\geq 1-c_1e^{-c_0s}.
\end{align*}
\end{lem}

We shall denote by $S_t(v)$ the set of vertices at distance $t$ from $v$.
We introduce 
\begin{align}
\bar \alpha_n  = \frac{1}{n}\max_{i \in [n]} \sum_{j=1}^n Q(X_i,X_j)=\alpha + O ( n^{-\gamma}).
\end{align}

\begin{lem}\label{l:Scontrol2}
There exist $c_0, c_1 >0$ such that for all $s \geq 0$ and for any $w \in [n] \cup \vec E(V)$,    
\begin{align*}
    \mathbb{P} \left( \forall t \geq 0: S_{t} (w)   \leq  s \bar \alpha_n^t  \right)  \geq 1 -  c_1 e^{ - c_0 s}. 
\end{align*}
Consequently, for any $p\geq 1$, there exists $c >0$ such that 
\begin{align*}
    \mathbb{E} \max_{v \in [n] , t \geq 0} \left(  \frac{ S_{t} (v) }{ \bar \alpha^t_n} \right )^p \leq c (\log n)^p.
\end{align*}
\end{lem}

\begin{prop}
Let $\ell \sim   \kappa \log_\alpha n$ with $0 < \kappa < \gamma/5$. For any $ k\in [r_0]$, there exists $\rho'_k >0$ such that w.h.p.
\begin{align*}
    \frac 1 {\alpha n} \sum_{e \in \vec E}  \frac{P^2_{k,\ell} (e) }{ \mu_k ^{4 \ell}}  \to  \rho'_k.
\end{align*}
\end{prop}

\begin{defn}
We set 
\begin{align*}
    I_{k} (v) = \sum_{e : e_2 = v}  s \sqrt { n} \xi_k' (e),
\end{align*}
where $s = \sqrt{ \alpha \rho'_k}$ and $\rho'_k$ was defined above and $\xi'$ is an orientation of $\xi$.
\end{defn}

\begin{lem}\label{l:Idifference}
We have that w.h.p.
\begin{align*}
    \frac{1} n \sum_{v=1}^n  \ABS{  I_k (v)   -  \frac{I_{k,\ell} (v)}{\mu_k ^ {2 \ell} }}  = o(1).  
\end{align*}
\end{lem}

\subsection{Proof of Lemma \ref{l:cconverge}}
Recall that
\begin{align*}
    \mathcal{C}_{ij}:=\frac{1}{\sqrt{n}}\sum\limits_{\ell=1}^n B_i(V_\ell)f_j(X_\ell).
\end{align*}
Define $J_{k,\ell}$ as before with the transition kernel being $(1-\epsilon)Q$. We use $J_{k,\ell}(x)$ to emphasize that the branching process starts at $X(o)=x$. By Lemma \ref{l:weakconverge}, we have that $J_{k,\ell}(x)   \mu_k ^{-2\ell}  - \alpha  \mu_k f_k (x)/(\mu_k ^2 / \alpha -1)$ converges in $L_2$ to a centered variable $Y_{k,x}$ satisfying $\mathbb{E}  |Y_{k,x}| \leq C$ uniformly for $x\in [0,1]$. Define
\begin{align*}
    H_{k,\ell}=\mu_k^{-2\ell}\int_0^1 J_{k,\ell}(x) f_k(x)dx,
\end{align*}
and
\begin{align*}
    H_{k,\ell,2}&=\mu_k^{-4\ell}\int_{[0,1]^2} J_{k,\ell}(x) J_{k,\ell}(y) Q(x,y) dx dy.
\end{align*}
Then $\mathbb{E}[H_{k,\ell}]\rightarrow \alpha  \mu_k/(\mu_k ^2 / \alpha -1)$ and $\mathbb{E}[H_{k,\ell,2}]\rightarrow \mu_k(\alpha  \mu_k/(\mu_k ^2 / \alpha -1))^2$. Recall the definition of $P_{k,\ell}(e)$ above. Recall that
\begin{align*}
    I_{k,\ell}(v)=\sum_{e\in\vec{E}:e_2=v}P_{k,\ell}(e).
\end{align*}
We define
\begin{align*}
    G_{k,\ell}=\mu_k^{-2\ell}\frac{1}{n}\sum_{v\in V}I_{k,\ell}(v)f_k(X_v),
\end{align*}
and 
\begin{align*}
    G_{k,\ell,2}&=\mu_k^{-4\ell}\frac{1}{n^2}\sum_{u,v\in V}I_{k,\ell}(u)I_{k,\ell}(v)Q(X_u,X_v).
\end{align*}
Now, we apply Proposition \ref{p:indep}. Take $\tau(G,v)$ as $\mu_k^{-2\ell}I_{k,\ell}(v)f_k(X_v)$. Let
\begin{align*}
    N(v)=\max_{0\leq t\leq \ell} \max_{u\in (G,v)_t} \max_{s\leq 2\ell-t} (S_s(u)/\alpha^s).
\end{align*}
As we have $\mu_k^2>\alpha$, by definition of $I_{k,\ell}$,
\begin{align*}
    \tau(G,v) &\leq \alpha^{-\ell} \frac{M}{\sqrt{\alpha}}\sum_{e\in \vec{E},e_2=v} \left( \sum_{t=0}^{\ell-1}\sum_{f\in \mathcal{Y}_t(e)} N^2(v)\alpha^{t+1}\alpha^{\ell-t}\right)\\
    &=N^2(v)\alpha \frac{M}{\sqrt{\alpha}}\sum_{e\in \vec{E},e_2=v} \left( \sum_{t=0}^{\ell-1}\sum_{f\in \mathcal{Y}_t(e)} 1\right).
\end{align*}
Hence $\tau(G,v)\leq C N(v)^3 \alpha^{\ell} =:\varphi(G,v)$. Then by Lemma \ref{l:Scontrol2}, we have $\mathbb{E}[\max_v \varphi(G,v)^4]=O(\log(n)^{12} \alpha^{4\ell})$. By Lemma \ref{l:Scontrol1} the same holds for $\varphi(T,o)$. Therefore, by Proposition \ref{p:indep}, we have that 
\begin{align*}
    \mathbb{E}|G_{k,\ell}-\mathbb{E}H_{k,\ell}| = o(1).
\end{align*}
Similarly, we consider $|G_{k,\ell,2}-\mathbb{E}H_{k,\ell,2}|$, which equals
\begin{align*}
    &\left|\mu_k^{-4\ell}\frac{1}{n^2}\sum_{u,v\in V}I_{k,\ell}(u)I_{k,\ell}(v)Q(X_u,X_v)-\mathbb{E}H_{k,\ell,2}\right|\\
    \leq &\mu_k^{-2\ell} \frac{1}{n}\sum_{u\in V}I_{k,\ell}(u)\left|\left(  \mu_k^{-2\ell} \frac{1}{n}\sum_{v\in V}I_{k,\ell}(v) Q(X_u,X_v)- \mathbb{E}[\mu_k^{-2\ell}\int_{[0,1]} J_{k,\ell} Q(X_u,y) dy]   \right)\right|\\
    +& \left|\mu_k^{-2\ell} \frac{1}{n}\sum_{u\in V}I_{k,\ell}(u)\mathbb{E}[\mu_k^{-2\ell}\int_{[0,1]} J_{k,\ell} Q(X_u,y) dy]-\mathbb{E}[\mu_k^{-4\ell}\int_{[0,1]^2} J_{k,\ell}(x) J_{k,\ell}(y) Q(x,y) dx dy]\right|.
\end{align*}
We apply Proposition \ref{p:indep} to both terms and take $\varphi$ similarly as above. So we have that
\begin{align*}
    \mathbb{E}|G_{k,\ell,2}-\mathbb{E}H_{k,\ell,2}|\rightarrow 0,
\end{align*}
as $n$ goes to infinity. This implies that $G_{k,\ell}$ converges to $\alpha  \mu_k /(\mu_k ^2 / \alpha -1)$ in probability. Further, $G_{k,\ell,2}$ converges to $\mu_k(\alpha  \mu_k/(\mu_k ^2 / \alpha -1))^2$ in probability. Recall the definition of 
\begin{align*}
    I_k(v)=\sum_{e:e_2=v}s\sqrt{n}\xi'_k(e),
\end{align*}
where $s=\sqrt{\alpha \rho_k'}$ and $\rho_k'$ was defined above, and $\xi'_k$ being an orientation of $\xi_k$. Let $O_k\in \{-1,1\}$ such that $\xi'_k(e)=O_k\xi_k(e)$. By Lemma \ref{l:Idifference}, we have that
\begin{align*}
    \frac{1}{n}\sum_{v=1}^n|I_k(v)-\frac{I_{k,\ell}}{\mu_k^{2\ell}}|=o(1).
\end{align*}
This implies that
\begin{align*}
\left|O_k\mathcal{C}_{kk}-\frac{G_{k,\ell}}{s}\right|&=\left|O_k\mathcal{C}_{kk}-\frac{1}{sn}\sum_{v\in V}f_k(X_v)\frac{I_{k,\ell}}{\mu_k^{2\ell}}\right|\\
&=\left|\frac{1}{n}\sum_{v\in V}f_k(X_v)\sum_{e:e_2=v}\sqrt{n}\xi'_k(e)-\frac{1}{sn}\sum_{v\in V}f_k(X_v)\frac{I_{k,\ell}}{\mu_k^{2\ell}}\right|\\
&=\left|\frac{1}{sn}\sum_{v\in V}f_k(X_v)I_k(v)-\frac{1}{sn}\sum_{v\in V}f_k(X_v)\frac{I_{k,\ell}}{\mu_k^{2\ell}}\right|\\
&\leq \sup f_k \frac{1}{sn}\sum_{v\in V}\left|I_k(v)-\frac{I_{k,\ell}}{\mu_k^{2\ell}}\right|=o(1).
\end{align*}
Further,
\begin{align*}
\left|\sum_{i=1}^\infty\mu_i\mathcal{C}_{ki}^2-\frac{G_{k,\ell,2}}{s}\right|&= \left|\frac{1}{s^2n^2}\sum_{u,v\in V}Q(X_v,X_u)(I_k(v) -\frac{I_{k,\ell}(v)}{\mu_k^{2\ell}})(I_k(u) +\frac{I_{k,\ell}(u)}{\mu_k^{2\ell}})\right|\\
&\leq \frac{M}{s^2n^2}\sum_{u,v\in V}\left|I_k(v) -\frac{I_{k,\ell}(v)}{\mu_k^{2\ell}}\right|\left|I_k(u) +\frac{I_{k,\ell}(u)}{\mu_k^{2\ell}}\right|=o(1).
\end{align*}
Therefore, the statement follows.

\printbibliography

\end{document}